%% file: InfMulticolorArxiv.tex
\documentclass[11pt]{article}\usepackage{fullpage,amsmath,amsfonts,amsthm,graphicx,amssymb,textcomp,enumerate,multicol,rotating,mathtools}
\usepackage{pgffor}
\usepackage{hyperref}
\usepackage{ifthen}

\usepackage{color}
\usepackage[all,cmtip]{xy}
\usepackage{tikz, tikz-cd}
\usetikzlibrary{matrix,arrows,decorations.pathmorphing, decorations.markings}

\usepackage{infmultcolor-tikzshortcuts}
\usepackage{infmultcolor-macros}

\newtheorem{theorem}{Theorem}[section]
\newtheorem{corollary}[theorem]{Corollary}
\newtheorem{lemma}[theorem]{Lemma}
\newtheorem{proposition}[theorem]{Proposition}

\theoremstyle{definition}
\newtheorem{definition}[theorem]{Definition}

\theoremstyle{remark}

\theoremstyle{remark}
\newtheorem{remark}[theorem]{Remark}

\theoremstyle{remark}

\newcommand{\N}{\mathbb{N}}
\newcommand{\Z}{\mathbb{Z}}
\newcommand{\CC}{\mathbb{C}}

\newcommand{\C}{\mathcal{C}}

\newcommand{\F}{\mathcal{F}}

\newcommand{\FT}{\mathcal{F}} 

\newcommand{\KC}[1]{C_N\left( #1 \right)}
\newcommand{\hkc}{\h_{C_N}}
\newcommand{\KCnonzero}[1]{C_N^{\geq 1}\left( #1 \right)}
\newcommand{\KCsimp}[1]{\widetilde{\KC{ #1 }}}

\newcommand{\infFT}{\mathcal{A}} 
\newcommand{\B}{\mathcal{B}} 
\newcommand{\BB}{\tilde{\B}} 
\newcommand{\Btwocol}[1][\gamma]{\B^{(#1)}_{(#1')}}
\newcommand{\BBtwocol}[1][\gamma]{\BB^{(#1)}_{(#1')}}
\newcommand{\Bn}{\mathfrak{B}_n} 
\newcommand{\BnGen}{\mathfrak{G}_n} 
\newcommand{\Bminus}[3]{#1 \setminus #1^{#2}_{#3}} 
\newcommand{\Bminuss}[5]{#1 \setminus \left(#1^{#2}_{#3},\dots,#1^{#4}_{#5}\right)} 
\newcommand{\Bminusss}[5]{#1 \setminus \left(#1^{#2}_{#3},#1^{#4}_{#5}\dots\right)} 

\newcommand{\Ach}{\mathbf{A}}
\newcommand{\Bch}{\mathbf{B}}
\newcommand{\Cch}{\mathbf{C}}
\newcommand{\Xch}{\mathbf{X}}
\newcommand{\Ych}{\mathbf{Y}}
\newcommand{\Mch}{\mathbf{M}}

\newcommand{\im}{\mathrm{Im}}
\newcommand{\cone}{\mathrm{Cone}}
\newcommand{\Mcone}{\mathrm{Mcone}}
\newcommand{\h}{\mathrm{h}}
\newcommand{\q}{\mathrm{q}}

\newcommand{\cs}[1]{|#1|_{\mathrm{c}}} 

\newcommand{\hspsymbol}[2]{\hspace{#1}#2\hspace{#1}}
\newcommand{\qq}[1]{\quad #1 \quad}

\newcommand{\td}[1]{\widetilde{#1}}

\newcommand{\webC}{\mathbf{NWeb}}
\newcommand{\foamC}{\mathbf{NFoam}}
\newcommand{\qbinom}[2]{{#1 \brack #2}}

\newcommand{\sln}{\mathfrak{sl}_N}

\newcommand{\qsln}{\mathcal{U}_q(\sln)}

\tikzstyle directed=[postaction={decorate,decoration={markings,
    mark=at position #1 with {\arrow{>}}}}]
\tikzstyle rdirected=[postaction={decorate,decoration={markings,
    mark=at position #1 with {\arrow{<}}}}]

\tikzset{->-/.style={decoration={
  markings,
  mark=at position #1 with {\arrow{>}}},postaction={decorate}}}
  
\tikzset{-<-/.style={decoration={
  markings,
  mark=at position #1 with {\arrow{<}}},postaction={decorate}}}

\def\centerarc[#1](#2)(#3:#4:#5)
    { \draw[#1] ($(#2)+({#5*cos(#3)},{#5*sin(#3)})$) arc (#3:#4:#5); }

\title{Khovanov-Rozansky homology for infinite multi-colored braids}
\author{Michael Willis \\
Department of Mathematics, UCLA\\
\href{mailto:msw188@ucla.edu}{\texttt{msw188@ucla.edu}}}

\begin{document}

\maketitle

\begin{abstract}
We define a limiting $\sln$ Khovanov-Rozansky homology for semi-infinite positive multi-colored braids, and we show that this limiting homology categorifies a highest-weight projector for a large class of such braids.  This effectively completes the extension of Cautis' similar result for infinite twist braids, begun in our earlier papers with Islambouli and Abel.  We also present several similar results for other families of semi-infinite and bi-infinite multi-colored braids.
\end{abstract}

\section{Introduction}
The Jones-Wenzl projector $P_n$ \cite{Wenzl} is a special idempotent element of the Temperley-Lieb algebra representing a highest-weight projector in the representation theory of $\mathcal{U}_q(\mathfrak{sl}_2)$, used in particular to define WRT-invariants for 3-manifolds (see, for example, \cite{KL}).  Similarly we have analogous highest-weight projectors in the representation theory of $\mathcal{U}_q(\mathfrak{sl}_N)$ for all $N$.  A sequence of papers by various authors \cite{Rose12, Cau12, Hog15} showed that, in all of these cases, such highest-weight projectors could be categorified via infinite chain complexes associated to the stable limiting Khovanov-Rozansky complex of infinite full twists (including the case as $N\rightarrow\infty$ corresponding to Khovanov-Rozansky HOMFLY-PT homology).  Indeed similar statements hold when we consider infinite full twists where we allow ourselves to color the strands with any natural number less than $N$, corresponding to irreducible anti-symmetric representations of $\mathcal{U}_q(\mathfrak{sl}_N)$.

A natural question for all of these cases one might ask would be what happens if we consider the limiting Khovanov-Rozansky complex of some other infinite braid.  Together with Islambouli in \cite{MWGI} and Abel in \cite{MAMW}, the author has shown that, for \emph{all} (semi-)infinite \emph{uni-colored} braids $\BB$ that are both positive and complete (a braid is complete if each braid group generator appears infinitely many times; intuitively, this means that $\BB$ contains all the crossings necessary to build the infinite full twist $\FT^\infty$), the limiting Khovanov-Rozansky complex $\KC{\BB}$ is chain homotopy equivalent to $\KC{\FT^\infty}$, the limiting complex of the infinite full twist.  Thus we have the following imprecise theorem (see the original papers for the precise versions).
\begin{theorem}[\cite{MWGI} Theorem 1.1, \cite{MAMW} Theorem 1.1]
\label{thm:old inf braid result}
All positive complete semi-infinite uni-colored braids categorify corresponding highest weight projectors.
\end{theorem}

The goal of this paper is to remove the ``uni-colored" restriction from Theorem \ref{thm:old inf braid result}.  Doing so will require a new restriction on the types of braids that are considered.  The imprecise version is presented here (the precise version will be stated in Section \ref{sec:proving main thm}).

\begin{theorem}
\label{thm:imprecise main thm}
All positive color-complete semi-infinite braids categorify corresponding highest weight projectors.
\end{theorem}

We will define color-completeness precisely in Section \ref{sec:color complete}, but intuitively this will mean that the braid $\BB$ will contain all of the `properly colored' crossings necessary to build the `properly colored' infinite full twist $\FT^\infty$.  Figure \ref{fig:noncolorcomplete ex} shows an example of a braid which is complete but not color-complete, and it can be seen why such a restriction might be expected.  Since the strands colored $i$ and $k$ in that example never interact, it seems plausible that the resulting limiting complex may not be related to the complex of the infinite full twist in which the strands colored $i$ and $k$ must twist around each other infinitely often.

\begin{figure}
\centering
\begin{tikzpicture}[x=.5cm,y=-.5cm]
\node at (0,0){$i$};
\node at (1,0){$j$};
\node at (2,0){$k$};
\Bsigma[1]{3}{1}
\Bsigma[2]{3}{1}
\Bsigma[3]{3}{2}
\Bsigma[4]{3}{2}
\Bsigma[5]{3}{1}
\Bsigma[6]{3}{1}
\Bsigma[7]{3}{2}
\Bsigma[8]{3}{2}
\node at (1,10){$\vdots$};
\end{tikzpicture}
\caption{An example of a positive semi-infinite braid (assuming this pattern of crossings continues indefinitely) that is complete if colors are ignored, but is not color-complete since the strands colored $i$ and $k$ never interact.  For a more precise definition, see Section \ref{sec:color complete}.}
\label{fig:noncolorcomplete ex}
\end{figure}
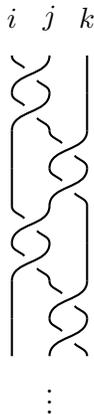

From Theorem \ref{thm:imprecise main thm} we will be able to deduce corollaries for a variety of situations.  We will have results for semi-infinite braids with finitely many negative crossings (Corollary \ref{cor:fin many negs}) allowing for a version of Theorem \ref{thm:imprecise main thm} taking Reidemeister II moves into account (Corollary \ref{cor:main thm posneg}).  We also have results for certain non-color-complete braids (Corollary \ref{cor:horizontal splitting}).  All of these are similar to our previous results in \cite{MAMW,MWGI}.  We will also discuss bi-infinite braids, for which defining a coloring and a corresponding limiting complex requires some more care.

\begin{theorem}
\label{thm:imprecise main thm bi-inf}
All positive color-complete bi-infinite braids $\BB$ have limiting complexes of the form $P \otimes \Cch \otimes P'$ where $P$ and $P'$ indicate categorified projectors for two (possibly different) sequences of colors depending on the coloring of $\BB$, while $\Cch$ is a complex categorifying certain maps between the representations determined by the two color sequences.
\end{theorem}

The precise version of Theorem \ref{thm:imprecise main thm bi-inf} will be presented as Corollary \ref{cor:main thm bi-inf}.

The proof of Theorem \ref{thm:imprecise main thm} effectively generalizes the earlier proofs for different versions of Theorem \ref{thm:old inf braid result} in \cite{MWGI,MAMW}.  We present a short outline here.  Let $\B$ be a semi-infinite braid word representing the infinite braid $\BB$, and let $\B_\ell$ denote the finite `partial' braid word consisting of the first $\ell$ crossings in $\B$.  We view $\B$ as a limit of $\B_\ell$ as the length of the words $\ell\rightarrow\infty$.  Unlike in previous versions, this limit will be taken along a subsequence of lengths such that each $\B_\ell$ represents a so-called `color-pure' braid (where the sequences of colors at the top and bottom of the braid must match).  The color-completeness assumption will ensure we have maps roughly of the form $F_\ell:\KC{\B_\ell} \rightarrow \KC{\FT^{z(\ell)}}$ where $z$ is some non-decreasing function of $\ell$ such that $z\rightarrow\infty$ as $\ell\rightarrow\infty$ along the necessary subsequence.  Our goal then is to show that $F_\ell$ is an isomorphism in homological degrees below some bound $|F_\ell|_\h$ that grows arbitrarily large as $\ell\rightarrow\infty$.  This is accomplished by keeping track of homological shifts in the cone of $F_\ell$ (actually in a larger complex containing $\cone(\F_\ell)$) resulting from pulling `rungs' through crossings and full twists.  Despite the structural similarity to the earlier versions of the proof, there are significant differences in the details of this version of the construction which we will focus on: the color-completeness and color-purity requirements for $\BB$ and $\B_\ell$ respectively; the resulting care needed to properly construct the system maps, the map $F_\ell$, and its cone; and the slightly different combinatorics encountered when computing the necessary homological shifts.

\begin{remark}
The local nature of the arguments in this paper potentially lend themselves well to similar homology theories for links in other 3-manifolds.  For instance, we expect all of our results to hold for the link homologies defined in the recent papers \cite{Quef,QW} when applied to an infinite braid within a 3-ball.  We leave precise statements and proofs about such generalizations for future consideration.
\end{remark}

This paper is arranged as follows.  Section \ref{sec:HomAlg} will present a short review of the necessary homological algebra, focused on defining and comparing limits of inverse systems of complexes as well as manipulating complexes defined as multicones.  Section \ref{sec:KR background} will review the relevant definitions and manipulations of colored Khovanov-Rozansky homology for braids.  Section \ref{sec:main pf} is the technical heart of the paper, providing the proof of Theorem \ref{thm:imprecise main thm}.  Section \ref{sec:corollaries} will explore a variety of corollaries, including the handling of negative crossings and bi-infinite braids.

\subsection{Acknowledgments} The author would like to thank Gabriel Islambouli and Michael Abel for their help with earlier versions of the arguments; Paul Wedrich for some helpful conversations; Sucharit Sarkar and Ciprian Manolescu for various advice; and Matt Hogancamp for many very patient conversations on handling some of the more technical homological algebra results.  The author was partially funded by NSF grant DMS-1563615 while preparing this manuscript.

\section{Homological algebra background}\label{sec:HomAlg}
In order to properly state and prove our theorems about semi-infinite chain complexes coming from semi-infinite braids, we will need the following ideas from homological algebra.  Throughout this section, all complexes are taken over some additive category, and differentials are taken to increase homological grading by one.

\begin{definition}\label{def:basic hom alg notations}
We employ the following basic notations.
\begin{itemize}
\item Chain complexes shall be denoted by capital boldface letters, as in $\Ach$, with differential $d_A$ (the subscript will be omitted if no confusion is possible).
\item Superscripts on capital letters will indicate homological degree.  Thus $A^i$ will denote the objects of $\Ach$ in homological degree $i$, while $\Ach^{\geq i}$ will denote the subcomplex of $\Ach$ consisting of all objects in homological degrees greater than or equal to $i$.
\item Single terms within a complex shall often be denoted by lowercase greek letters; the notation $\alpha\in\Ach$ indicates that $\alpha$ is a term in the complex $\Ach$.
\item For $\alpha\in\Ach$, $\h_A(\alpha)$ shall denote the homological degree of the term $\alpha$ within $\Ach$.
\item $\h$ will also denote the homological shift functor, so that $\h^t\Ach$ indicates the complex $\Ach$ where all terms have been shifted upwards in homological degree by $t$.
\item The notation $f\sim g$ will indicate that two chain maps $f$ and $g$ are homotopic.  The notation $\Ach\simeq\Bch$ will indicate that two complexes $\Ach$ and $\Bch$ are chain homotopy equivalent.
\end{itemize}
\end{definition}

\subsection{Inverse systems and limits}\label{sec:InvSys}
The entirety of this section is based on definitions in \cite{RozInf}, and is taken nearly verbatim from the similar sections in \cite{MWGI,MAMW}.

\begin{definition}\label{def:homOrder}
Let $\Ach$ and $\Bch$ be chain complexes and suppose $f: \Ach \rightarrow \Bch $ is a chain map. Define the \emph{homological order} of $f$, which we denote by $|f|_\h$, to be the maximal degree $d$ for which the cone $\cone( \Ach \xrightarrow{f} \Bch )$ is chain homotopy equivalent to a complex $\Cch$ that is contractible below homological degree $d$.
\end{definition}

Roughly speaking, $|f|_\h$ denotes the maximal homological degree through which $f$ gives a homotopy equivalence on the truncated complexes of $\Ach$ and $\Bch$ up to degree $|f|_\h$.

\begin{definition}\label{def:invSys}
	An \emph{inverse system} of chain complexes is a sequence of chain complexes equipped with chain maps
	$$\{\Ach_k,f_k\}=\Ach_0 \xleftarrow{f_0} \Ach_1 \xleftarrow{f_1} \Ach_2 \xleftarrow{f_2} \cdots.$$
	An inverse system is called \emph{Cauchy} if $|f_k|_\h \to \infty$ as $k \to \infty$.
\end{definition}

\begin{definition} \label{def:invLim}
An inverse system $\{\Ach_k, f_k\}$ has a \emph{limit} (or \emph{inverse limit}), which we denote by $\Ach_\infty$ or $\lim \Ach_k$, if there exist maps $\tilde{f}_k: \Ach_\infty \to \Ach_k$ such that
\begin{itemize}
 	\item $f_k\circ \tilde{f}_{k+1}=\tilde{f}_k$ for all $k\geq 0$, and
 	\item $|\tilde{f}_k|_\h \to \infty$ as $k \to \infty$. 
\end{itemize}
\end{definition}

\begin{theorem}[\cite{RozInf} Propositions 3.7 and 3.12]\label{thm:Cauchy}
An inverse system $\{\Ach_k,f_k\}$ of chain complexes has a limit $\Ach_\infty$ if and only if it is Cauchy.
\end{theorem}

Based on Definition \ref{def:invLim}, it is easy to see (see \cite{RozInf}) that limits to inverse systems are unique if they exist.  From this it is also easy to prove the following lemma, which will be used mainly for concatenations of finite and infinite braid complexes.

\begin{lemma}\label{lem:Lim of product}
Given an inverse system $\{\Ach_k, f_k\}$ with limit $\Ach_\infty$ and a complex $\Bch$, there is a corresponding inverse system $\{(\Bch \otimes \Ach_k), (I_{\Bch} \otimes f_k )\}$ with inverse limit satisfying
\[\lim (\Bch \otimes \Ach_k) \simeq \Bch \otimes \Ach_\infty.\]
In other words, the limiting process commutes with the tensor product.
\end{lemma}
\begin{proof}
Apply the functor $\Bch \otimes (\cdot)$ to the diagram comprised of the inverse system $\{\Ach_k, f_k\}$ together with $\Ach_\infty$ and the maps $\tilde{f}_k$.  Note that homological orders of maps are preserved, and then appeal to the uniqueness of the limit.
\end{proof}

We conclude this section with a result from \cite{MWGI} which allows us to prove two inverse systems have equivalent limits. 

\begin{proposition}[\cite{MWGI} Proposition 2.13] \label{prop:compSys}
	Suppose $\{\Ach_k,f_k\}$ and $\{\Bch_{\ell},g_{\ell}\}$ are two Cauchy inverse systems with limits $\Ach_\infty$ and $\Bch_\infty$ respectively. Let $z(\ell)$ be a nondecreasing function on $\N$ such that $\lim_{\ell\rightarrow\infty}z(\ell)=\infty$. Suppose there are maps $F_\ell : \Bch_\ell \to \Ach_{z(\ell)}$ forming a commuting diagram with the system maps $f_k$ and $g_{\ell}$ for appropriate $k$ and $\ell$. If $|F_\ell|_\h \to \infty$ as $\ell \to \infty$ then $\Ach_\infty \simeq \Bch_\infty$.
\end{proposition}

In Figure \ref{fig:compare inv sys} we include a diagram to better explain the situation in Proposition \ref{prop:compSys}.  In principal we could allow cases where $z(\ell+1)>z(\ell)+1$, necessitating the use of multiple system maps $f_k$ in the commutation required; in the cases of interest in this paper, however, this will never be the case and $z$ will only ever increase by one or not at all.

\begin{figure}[ht] 
\centering
\includegraphics[scale=.35]{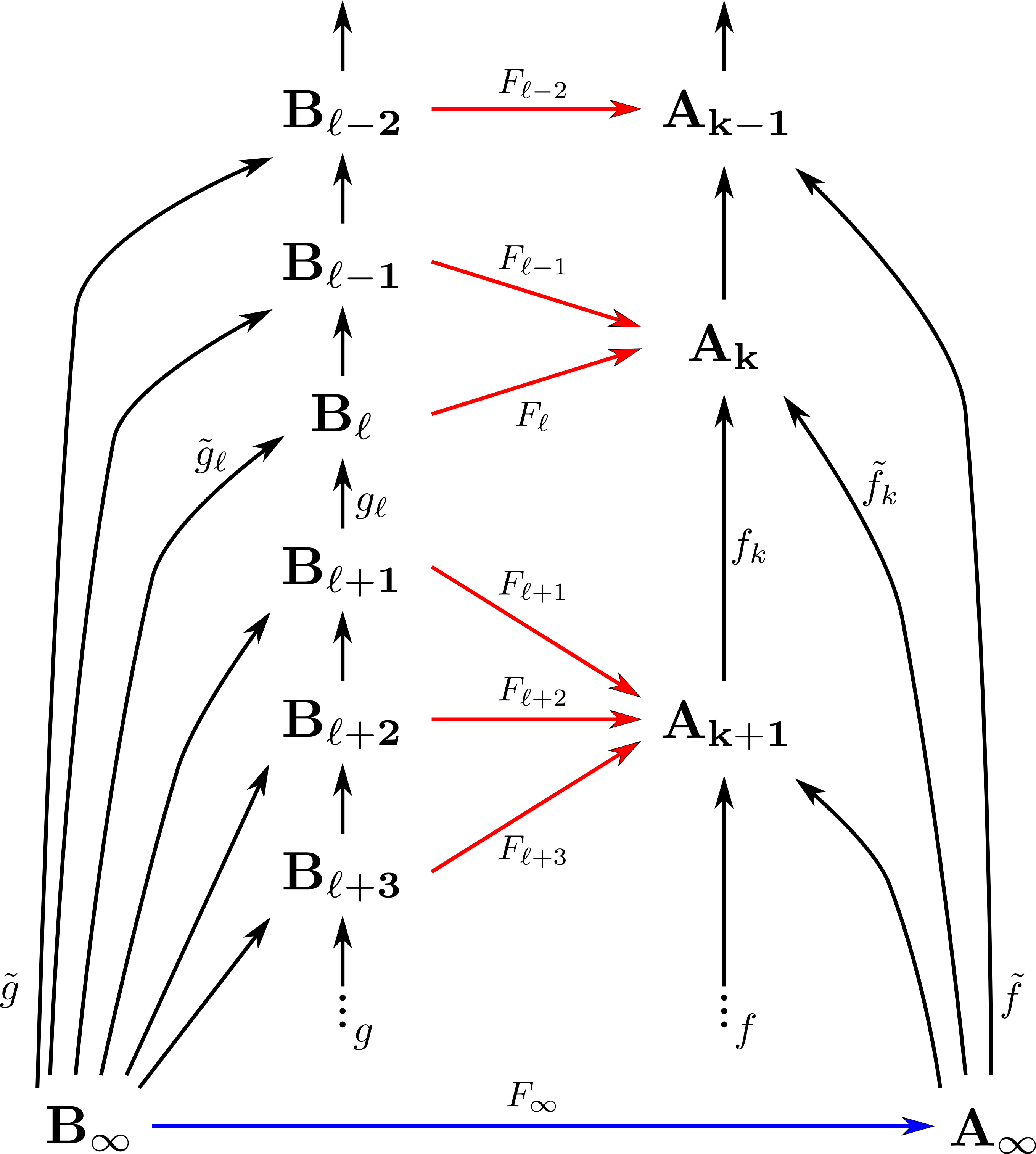}
\caption{The diagram for Proposition \ref{prop:compSys}. In this diagram both $\{\Ach_k,f_k\}$ and $\{\Bch_{\ell},g_{\ell}\}$ are Cauchy inverse systems. The limits $\Ach_\infty$ and $\Bch_\infty$ and the maps $\tilde{f},\tilde{f}_k,\tilde{g},\tilde{g_{\ell}}$ all exist from Theorem \ref{thm:Cauchy}. If we can find the maps $F_\ell$, then Theorem \ref{thm:Cauchy} also provides the map $F_\infty$. We simply need to show that $|F_\ell|_\h \to \infty$ as $\ell\to\infty$ to prove that $F_\infty$ is a homotopy equivalence.}
\label{fig:compare inv sys}
\end{figure}

 \subsection{Multicone complexes} \label{sec:multicones}
Many of the complexes that we will be concerned with in this paper are most easily understood and manipulated using a generalization of the usual cone construction as follows (a large portion of this section is taken nearly verbatim from \cite{MWs2s1}).

\begin{definition}\label{def:gen multicone def}
Suppose we are given the following data in a fixed category of chain complexes over some additive category:
\begin{itemize}
\item A finite index set $\C$ with a $\Z$-grading $\h_{\C}:\C\rightarrow\Z$.
\item For all $i\in\C$, a chain complex $\Ach_i$ with differential $d_i$.
\item For all $i,j\in\C$, a map (not necessarily a chain map) $f_{ij}:\Ach_i\rightarrow \h^{\h_{\C}(j)-\h_{\C}(i) -1} \Ach_j$ satisfying
\begin{itemize}
	\item $f_{ii}:=d_i$,
	\item for all $j\neq i$ in $\C$ with $\h_{\C}(j)\leq\h_{\C}(i)$, $f_{ij}:=0$, and
	\item for all $i,k\in\C$, $\sum_{j\in\C} f_{jk} f_{ij} = 0$.
\end{itemize}
\end{itemize}
Then we can form the \emph{multicone} 
\begin{equation}\label{eq:gen multicone eq}
\Mch=\underset{i,j\in\C}{\Mcone} \left( \Ach_i \xrightarrow{f_{ij}} \h^{\h_{\C}(j)-\h_{\C}(i) - 1} \Ach_j \right)
\end{equation}
 which is a chain complex $\Mch$ whose terms are the direct sum of all of the terms of the complexes $\Ach_i$
\[\Mch:=\bigoplus_{i\in\C} \Ach_i\]
and whose differential $d_M$ is the sum of all of the maps $f_{ij}$
\[d_M:=\sum_{i,j\in\C} f_{ij}.\]
For a term $\alpha\in \Ach_i\subset \Mch$, we determine the homological grading as the sum of the contributions of viewing $\alpha$ in $\Ach_i$ and viewing $\Ach_i$ in $\C$
\begin{equation}\label{eq:gen multicone hom deg}
\h_M(\alpha):= \h_{A_i}(\alpha) + \h_{\C}(i).
\end{equation}
\end{definition}
The reader may verify that this definition gives a well defined chain complex.  When $\h_\C(j)-\h_\C(i)=1$, the maps $f_{ij}$ assemble to define chain maps; when $\h_\C(j)-\h_\C(i)=2$, the maps $f_{ij}$ assemble to form null-homotopies for the compositions of any two of these chain maps; and so on.  Because the original system $\C$ was finite, this process must eventually end, and so of course the sum in the definition of $d_M$ is finite.  When $\C=\{1,2\}$ and we have the single chain map $f_{12}:\Ach_1\rightarrow \Ach_2$, this construction recovers the usual cone on $f_{12}$.

\begin{remark}
We employ the term `multicone' in Definition \ref{def:gen multicone def} following \cite{RozS1S2}.  A complex built in this manner is also often referred to as a \emph{totalization} or \emph{convolution} of a \emph{twisted complex}.  See for instance \cite{BK}.
\end{remark}

Note that \emph{any} finite chain complex $\Bch$ can be represented as a multicone by declaring that $\C$ is indexed by the terms in $\Bch$ while the maps $f_{i(i+1)}$ are given by the differentials of $\Bch$.  Meanwhile all of the maps $f_{ij}$ with $\h_\C(j)-\h_\C(i)\geq2$ are zero maps (ie no homotopies are needed).

The following proposition gives us our main tool for manipulation of multicone complexes.

\begin{proposition}\label{prop:gen multicone equiv}
Given a chain complex presented as a multicone $\Mch$ as in Equation \eqref{eq:gen multicone eq}, and given chain homotopy equivalences $\iota_i:\Ach_i\rightarrow \Ach_i'$ for each $i\in\C$, there exist maps $f_{ij}'$ such that we can form the multicone
\[\Mch':=\underset{i,j\in\C}{\Mcone} \left( \Ach'_i \xrightarrow{f_{ij}'} \h^{\h_{\C}(j)-\h_{\C}(i) - 1} \Ach_j' \right)\]
that is chain homotopy equivalent to $\Mch$:
\[\Mch\simeq \Mch'.\]
\end{proposition}
\begin{proof} This is a standard result generalizing the fact that the homotopy category of complexes over an additive category is triangulated (Proposition 2 in \cite{BK}).
\end{proof}

\section{Colored $\mathfrak{sl}(N)$ link homology background}\label{sec:KR background}

\subsection{The $\mathfrak{sl}(N)$ foam category and colored Khovanov-Rozansky homology}\label{sec:KR definitions}

Colored Khovanov-Rozansky homology was first constructed independently by Wu \cite{Wu} and Yonezawa \cite{Yonez}. This homology theory generalizes the original construction of Khovanov and Rozansky \cite{KR} and categorifies the colored $SL(N)$ polynomial when coloring components by fundamental representations. Queffelec and Rose gave a combinatorial/geometric construction of colored Khovanov-Rozansky homology in terms of `webs' and `foams' \cite{QR}. It is this construction which we will briefly recall here (again, a large portion of this review is taken nearly verbatim from \cite{MAMW}). 

We begin with the category $\webC$, as described by Cautis, Kamnitzer, and Morrison in \cite{CKM}.  The objects of $\webC$ are given by sequences $\gamma = (\gamma_1,...,\gamma_\ell)$ where $\ell > 0$ and $\gamma_i \in \{0,1,...,N\}$ called \emph{colorings}, together with a zero object.  The 1-morphisms are formal sums of upward oriented trivalent graphs with edges labeled by integers in the same coloring set $\{0,1,...,N\}$.  At any vertex, the labels of the two edges of similar orientation (incoming or outgoing) are required to sum up to give the label of the third edge.  Such graphs are generated by the basic graphs in Figure \ref{fig:basicWebs}.  There is also a set of local relations for such trivalent graphs; we list a few of the most important ones in Figure \ref{fig:webRel}, and refer the reader to \cite{CKM} for a complete list.


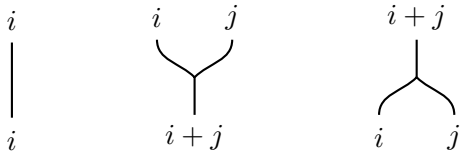
\begin{figure}[ht]
\[
\begin{tikzpicture} 
\draw[thick] (0,0) -- (0,1);
\node at (0,1.3){$i$};
\node at (0,-.3){$i$};
\end{tikzpicture}
\qquad
\qquad
\begin{tikzpicture} 
\draw[thick] (0,0) -- (0,.5) to[out=45,in=-90] (.5,1)
(-.5,1) to[out=-90,in=90+45] (0,.5);
\node at (-.5,1.3){$i$};
\node at (.5,1.3){$j$};
\node at (0,-.3){$i+j$};
\end{tikzpicture}
\qquad
\qquad
\begin{tikzpicture} 
\draw[thick] (0,1) -- (0,.5) to[out=-45,in=90] (.5,0)
(-.5,0) to[out=90,in=180+45] (0,.5);
\node at (-.5,-.3){$i$};
\node at (.5,-.3){$j$};
\node at (0,1.3){$i+j$};
\end{tikzpicture}
\]

\caption{Basic $\sln$-webs (all edges are oriented upwards).}
\label{fig:basicWebs}	
\end{figure}

\begin{figure}[ht]
\[
\begin{tikzpicture}[x=2em,y=2em,baseline={([yshift=-1ex]current bounding box.center)}]
\draw[thick] (0,0) -- (0,.5) to[out=45,in=-90] (.5,1) to[out=90,in=-45] (0,1.5) -- (0,2)
(0,1.5) to[out=180+45,in=90] (-.5,1) to[out=-90,in=90+45] (0,.5);
\smnode at (-.8,1){$i$};
\smnode at (.8,1){$j$};
\smnode at (0,-.3){$i+j$};
\smnode at (0,2.3){$i+j$};
\end{tikzpicture}
\qq{=} \qbinom{i+j}{i} 
\begin{tikzpicture}[x=2em,y=2em,baseline={([yshift=-1ex]current bounding box.center)}]
\draw[thick](0,0)--(0,2);
\smnode at (0,-.3){$i+j$};
\smnode at (0,2.3){$i+j$};
\end{tikzpicture}
\]

\[
\begin{tikzpicture}[x=2em,y=-2em,baseline={([yshift=-1ex]current bounding box.center)}]
\URrungBL[0]{2}{1}{}
\ULrungBR[1]{2}{1}{$\ell$}
\smnode at (0,-.3){$i+k$};
\smnode at (1,-.3){$j$};
\smnode at (0,2.3){$i$};
\smnode at (1,2.3){$j+k$};
\end{tikzpicture}
\qq{=} \sum_{p=\max(0,k)}^\ell \qbinom{j-i}{\ell-p}
\begin{tikzpicture}[x=2em,y=-2em,baseline={([yshift=-1ex]current bounding box.center)}]
\ULrungBR[0]{2}{1}{$p$}
\URrungBL[1]{2}{1}{}
\smnode at (0,-.3){$i+k$};
\smnode at (1,-.3){$j$};
\smnode at (0,2.3){$i$};
\smnode at (1,2.3){$j+k$};
\end{tikzpicture}
\]

\[[n] = \dfrac{q^n-q^{-n}}{q-q^{-1}}, \quad [n]! = [n][n-1]\cdots[1],\quad \qbinom{n}{k} = \dfrac{[n]!}{[n-k]![k]!}\]
\caption{Selected relations for $\sln$-webs (all edges are oriented upwards).  See \cite{CKM} for a full list.}
\label{fig:webRel}
\end{figure}

We interpret any such graph as a mapping between the coloring at the bottom to the coloring at the top. These graphs are called \emph{$\sln$-webs} due to their relation to the representation theory of $\qsln$. Sometimes we will omit edges labeled by $0$ and $N$, but allowing these labels in the definition will make later definitions easier to write. By convention, we will also allow edges labeled by integers larger than $N$ for the sake of later definitions. However, any web with such an edge will be set equal to the zero web (that is, the unique morphism factoring through the zero object). 

We now move on to the $2$-category $\foamC$.  The objects and 1-morphisms of $\foamC$ are the same as those for $\webC$, except that the 1-morphisms are considered as direct sums rather than genuine sums, and the relations between such 1-morphisms are discarded for the moment.  The 2-morphisms are matrices of labeled singular cobordisms between $\sln$-webs, called $\sln$-foams. These cobordisms are generated by the basic cobordisms in Figure \ref{fig:basicFoams}.

\begin{figure}[ht]
\include{BasicFoams}
\caption{The basic generating $\sln$-foams in $\foamC$ (Image from \cite{QR}).}
\label{fig:basicFoams}
\end{figure}

Similar to the convention for $\sln$-webs, we will interpret $\sln$-foams as mapping from the bottom boundary to the top boundary. Each facet of an $\sln$-foam is labeled with an element of $\{0,1,...,N\}$.  Any facet whose boundary is shared with an edge of a web must share the same label as that edge of the web.  We allow decorations $\bullet_p$ on the facets of the foams where $p$ is a symmetric polynomial in a number of variables equal to the label of the facet.  There also exists a set of local relations for these $\sln$-foams which allow for a lifting of the web relations in $\webC$ (such as those in Figure \ref{fig:webRel}) to 2-isomorphisms between the corresponding 1-morphisms in $\foamC$.  The reader should consult \S3 of \cite{QR} for more details. 

The 2-morphisms in $\foamC$ satisfy an adjunction equality in the following sense.  If $\delta_a$ and $\delta_b$ are two $\sln$-webs (1-morphisms in $\foamC$), then we have an isomorphism of 2-morphism spaces
\begin{equation}\label{eq:2mor duality}
\hom_2(\delta_a,\delta_b) \cong \hom_2(\delta_a\cdot\delta_b^\wedge,I).
\end{equation}
Here $\hom_2(\cdot,\cdot)$ denotes the 2-morphism spaces, $I$ denotes the relevant colored identity diagram, and $\delta_b^\wedge$ denotes the \emph{dual} of $\delta_b$ obtained from $\delta_b$ by reflection about a horizontal line.  The $\cdot$ notation in Equation \eqref{eq:2mor duality} indicates contatenation of $\sln$-webs, which plays the role of (monoidal) tensor product.  An illustration of this isomorphism will be provided in Figure \ref{fig:duality iso} for the more general setting of complexes over $\foamC$.

The category $\foamC$ also admits a grading. We will note the gradings in cases that it is necessary, but will once again refer the reader to \cite{QR} for more complete information. All chain complexes of foams are assumed to have degree 0 differentials. We will denote grading shifts in $\foamC$ with the notation $\q^k$ for a grading shift upwards by $k$. 

To any tangle diagram $T$ whose components are labeled by elements of $\{0,1,...,N\}$, we can associate a chain complex in $\foamC$, which we will denote by $\KC{T}$. The homotopy equivalence type of $\KC{T}$ is an isotopy invariant of the tangle $T$. In diagrams and figures, we will often omit the notation $\KC{\cdot}$ and simply draw the corresponding tangle or web unless there is a chance for confusion. In this text we will exclusively focus on the case that the tangle $T$ is actually a braid.

To build a chain complex in $\foamC$ for a colored braid (by convention we orient all of the strands upwards), we construct basic chain complexes for each crossing. Suppose that $i \leq j$, then

\begin{equation}\label{eq:posxDef}
\begin{tikzpicture}[x=2em,y=-2.5em,baseline={([yshift=-2ex]current bounding box.center)}]
\Bsigma[0]{2}{1}
\smnode at (0,-.3){$j$};
\smnode at (1,-.3){$i$};
\smnode at (0,1.3){$i$};
\smnode at (1,1.3){$j$};
\end{tikzpicture}
:= \qquad
\begin{tikzpicture}[x=2em,y=-1.5em,baseline={([yshift=-2ex]current bounding box.center)}]
\smnode at (0,-.3){$j$};
\smnode at (1,-.3){$i$};
\smnode at (0,2.3){$i$};
\smnode at (1,2.3){$j$};
\ULrungBR[0]{2}{1}{}
\URrungBL[1]{2}{1}{$0$}
\end{tikzpicture}
\xlongrightarrow{d_0} \q
\begin{tikzpicture}[x=2em,y=-1.5em,baseline={([yshift=-2ex]current bounding box.center)}]
\smnode at (0,-.3){$j$};
\smnode at (1,-.3){$i$};
\smnode at (0,2.3){$i$};
\smnode at (1,2.3){$j$};
\ULrungBR[0]{2}{1}{}
\URrungBL[1]{2}{1}{$1$}
\end{tikzpicture}
\xlongrightarrow{d_1} \cdots \xlongrightarrow{d_{i-1}}
\q^i
\begin{tikzpicture}[x=2em,y=-1.5em,baseline={([yshift=-2ex]current bounding box.center)}]
\smnode at (0,-.3){$j$};
\smnode at (1,-.3){$i$};
\smnode at (0,2.3){$i$};
\smnode at (1,2.3){$j$};
\ULrungBR[0]{2}{1}{}
\URrungBL[1]{2}{1}{$i$}
\end{tikzpicture}
\end{equation}

\begin{equation}\label{eq:negxDef}
\begin{tikzpicture}[x=2em,y=-2.5em,baseline={([yshift=-2ex]current bounding box.center)}]
\Bsigmainv[0]{2}{1}
\smnode at (0,-.3){$j$};
\smnode at (1,-.3){$i$};
\smnode at (0,1.3){$i$};
\smnode at (1,1.3){$j$};
\end{tikzpicture}
:= \qquad
\begin{tikzpicture}[x=2em,y=-1.5em,baseline={([yshift=-2ex]current bounding box.center)}]
\smnode at (0,-.3){$j$};
\smnode at (1,-.3){$i$};
\smnode at (0,2.3){$i$};
\smnode at (1,2.3){$j$};
\ULrungBR[0]{2}{1}{}
\URrungBL[1]{2}{1}{$i$}
\end{tikzpicture}
\xlongrightarrow{d'_{i-1}} \q
\begin{tikzpicture}[x=2em,y=-1.5em,baseline={([yshift=-2ex]current bounding box.center)}]
\smnode at (0,-.3){$j$};
\smnode at (1,-.3){$i$};
\smnode at (0,2.3){$i$};
\smnode at (1,2.3){$j$};
\ULrungBR[0]{2}{1}{}
\URrungBL[1]{2}{1}{$i-1$}
\end{tikzpicture}
\xlongrightarrow{d'_{i-2}} \cdots \xlongrightarrow{d'_0}
\q^i
\begin{tikzpicture}[x=2em,y=-1.5em,baseline={([yshift=-2ex]current bounding box.center)}]
\smnode at (0,-.3){$j$};
\smnode at (1,-.3){$i$};
\smnode at (0,2.3){$i$};
\smnode at (1,2.3){$j$};
\ULrungBR[0]{2}{1}{}
\URrungBL[1]{2}{1}{$0$}
\end{tikzpicture}
\end{equation}
Trivalent graphs of the form illustrated on the right hand side of Equations \eqref{eq:posxDef} and \eqref{eq:negxDef} will be referred to as \emph{ladders}; the nearly horizontal slanted edges will be referred to as \emph{rungs}.  The term furthest to the left is taken to have homological degree zero, and the differential has homological degree $1$.  The symbol $\q$ is used to denote a shift in the internal quantum grading; we will often omit this shift as it will have no bearing on our arguments.  We remark that the labeled edges determine all edges in each web, and that certain webs may be zero webs if they have a label larger than $N$.

The maps $d_k$ are degree $0$ foams as specified in \cite{QR}. The maps $d_k'$ are the same foams, but reflected to switch the source and target webs. Finally, if $i > j$, then we reflect each web around a vertical axis and perform the analogous transformation to the foams $d_k$ and $d'_k$.

We note that our conventions differ from those of \cite{QR}. In particular, our $\KC{\posxji}$ differs from their definition of $\KC{\posxji}$ by a shift of $\h^i\q^i$ (where $i \leq j$). However, our convention makes studying stabilization behavior more straightforward.  We remark that under our convention, Reidemeister I and II moves hold only up to a shift (see \cite{MAMW}).  In particular, a Reidemeister II move incurs a shift of $\h^t$ where $t$ is the minimum between the two colors involved.  Reidemeister I shifts will not be relevant for us in this paper.

In order to complete the definition of $\KC{T}$ for a positive braid $T$, we take the planar tensor product of the various $\KC{\posxji}$ and $\KC{\negxji}$ for each crossing in $T$ in the usual sense of Bar-Natan's planar algebras \cite{BN}; that is, we stitch together the various webs and foams while taking the tensor product of the corresponding complexes.  Thus, if we choose to apply Equations \eqref{eq:posxDef} and \eqref{eq:negxDef} to only certain crossings within a given braid, we will construct multicone complexes where each term involves ladder diagrams within a larger braid-like diagram (see Figure \ref{fig:braid as mcone ex} for an example illustrating this notion and our notation).  We refer to such diagrams as \emph{web-braid diagrams} ($\sln$-webs generically embedded in a 3-ball with $n$ incoming and $n$ outgoing strands, projected to the plane with no turnbacks present).

\begin{figure}
\begin{align*}
\begin{tikzpicture}[x=2em,y=-2em,baseline={([yshift=-2ex]current bounding box.center)}]
\Bsigma[0]{2}{1}
\Bsigma[1]{2}{1}
\smnode at (0,-.3){i};
\smnode at (1,-.3){j};
\smnode at (0,2.3){i};
\smnode at (1,2.3){j};
\end{tikzpicture}
\qq{&\simeq}
\begin{tikzpicture}[x=2em,y=-2em,baseline={([yshift=-2ex]current bounding box.center)}]
\Bsigma[0]{2}{1}
\ULrungBR[1]{2}{1}{}
\URrungBL[2]{2}{1}{$0$}
\smnode at (0,-.3){i};
\smnode at (1,-.3){j};
\smnode at (0,3.3){i};
\smnode at (1,3.3){j};
\end{tikzpicture}
\qq{\longrightarrow}
\begin{tikzpicture}[x=2em,y=-2em,baseline={([yshift=-2ex]current bounding box.center)}]
\Bsigma[0]{2}{1}
\ULrungBR[1]{2}{1}{}
\URrungBL[2]{2}{1}{$1$}
\smnode at (0,-.3){i};
\smnode at (1,-.3){j};
\smnode at (0,3.3){i};
\smnode at (1,3.3){j};
\end{tikzpicture}
\qq{\longrightarrow}
\begin{tikzpicture}[x=2em,y=-2em,baseline={([yshift=-2ex]current bounding box.center)}]
\Bsigma[0]{2}{1}
\ULrungBR[1]{2}{1}{}
\URrungBL[2]{2}{1}{$2$}
\smnode at (0,-.3){i};
\smnode at (1,-.3){j};
\smnode at (0,3.3){i};
\smnode at (1,3.3){j};
\end{tikzpicture}
\qq{\longrightarrow}
\cdots
\qq{\longrightarrow}
\begin{tikzpicture}[x=2em,y=-2em,baseline={([yshift=-2ex]current bounding box.center)}]
\Bsigma[0]{2}{1}
\ULrungBR[1]{2}{1}{}
\URrungBL[2]{2}{1}{$i$}
\smnode at (0,-.3){i};
\smnode at (1,-.3){j};
\smnode at (0,3.3){i};
\smnode at (1,3.3){j};
\end{tikzpicture}
\\
\qq{&\simeq}
\underset{\delta_x,\delta_y\in\KC{\tau_{(ij)}}}{\Mcone}\left( 
\begin{tikzpicture}[x=2em,y=-2em,baseline={([yshift=-2ex]current bounding box.center)}]
\Bsigma[0]{2}{1}
\Bbox[1]{2}{1}{2}{$\delta_x$}
\end{tikzpicture}
\qq{\longrightarrow}
\begin{tikzpicture}[x=2em,y=-2em,baseline={([yshift=-2ex]current bounding box.center)}]
\Bsigma[0]{2}{1}
\Bbox[1]{2}{1}{2}{$\delta_y$}
\end{tikzpicture}
\right)
\end{align*}
\caption{The simplest example of viewing $\KC{T}$ as a multicone along a chosen set of crossings (in this case, along one crossing denoted $\tau_{(ij)}$).  The second line illustrates the multicone notation we will use for such a complex.  The symbols $\delta_x,\delta_y$ denote diagrams appearing in the complex $\KC{\tau_{(ij)}}$.  We omit the $\KC{\cdot}$ notation and the various grading shifts for simplicity.}
\label{fig:braid as mcone ex}
\end{figure}
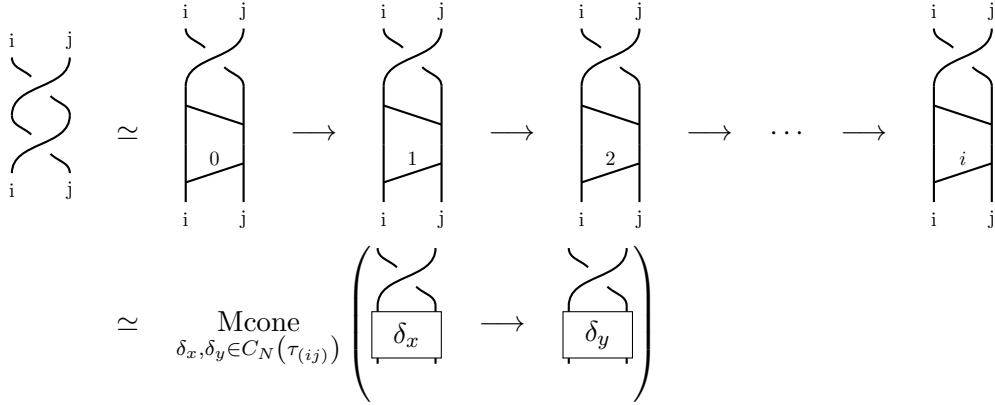

The category of complexes over $\foamC$ admits an adjunction equality much like that of Equation \eqref{eq:2mor duality}.  Given two tangles (or more generally, two web-braid diagrams) $T_1$ and $T_2$, the space of chain maps $\hom(\KC{T_1},\KC{T_2})$ between their corresponding complexes is itself a complex, and satisfies
\begin{equation}\label{eq:chain cx duality}
\hom(\KC{T_1},\KC{T_2}) \simeq \hom(\KC{T_1\cdot T_2^\wedge},\KC{I})
\end{equation}
where once again $T_2^\wedge$ is obtained from $T_2$ by a horizontal reflection, and $I$ is the relevant (colored) identity tangle.  We illustrate this diagrammatically in Figure \ref{fig:duality iso}.

\begin{figure}
\[
\hom\left(
\quad
\begin{tikzpicture}[x=2em,y=-2em,baseline={([yshift=-2ex]current bounding box.center)}]
\Bboxst[0]{$B$}
\smnode at (-.3,0){$\gamma$};
\smnode at (-.3,1){$\gamma'$};
\end{tikzpicture}
\qq{,}
\begin{tikzpicture}[x=2em,y=-2em,baseline={([yshift=-2ex]current bounding box.center)}]
\Bboxst[0]{$A$}
\smnode at (-.3,0){$\gamma$};
\smnode at (-.3,1){$\gamma'$};
\end{tikzpicture}
\quad
\right)
\qq{\simeq}
\hom\left(
\quad
\begin{tikzpicture}[x=2em,y=-2em,baseline={([yshift=-2ex]current bounding box.center)}]
\Bboxst[0]{$B$}
\Bboxst[1]{\scalebox{1}[-1]{$A$}}
\smnode at (-.3,0){$\gamma$};
\smnode at (-.3,1){$\gamma'$};
\smnode at (-.3,2){$\gamma$};
\end{tikzpicture}
\qq{,}
\begin{tikzpicture}[x=2em,y=-2em,baseline={([yshift=-2ex]current bounding box.center)}]
\foreach \j in {0,1}
 { \Bsigma[\j]{2}{0} }
\smnode at (-.3,0){$\gamma$};
\smnode at (-.3,2){$\gamma$};
\end{tikzpicture}
\quad
\right)
\]
\caption{A diagrammatic example of Equation \eqref{eq:chain cx duality} using two web-braid diagrams which we denote $B$ and $A$, both having coloring $\gamma$ at the top and coloring $\gamma'$ at the bottom.  The complex of chain maps between them is itself chain homotopy equivalent to the complex of maps from  $\KC{B\cdot A^\wedge}$ to the $\gamma$-colored identity tangle.  The $\KC{\cdot}$ notation has been omitted.}
\label{fig:duality iso}
\end{figure}

\subsection{Manipulating some basic complexes in $\foamC$}
Given a web-braid diagram, we may `slide' and/or `twist' rungs past various crossings (see Figures \ref{fig:rungslide} and \ref{fig:rungtwist}).  In such cases, the colors involved at the crossings that have been passed by will change and we will need to understand how this affects the resulting complex.  The following proposition is proved in \cite{MAMW}. 

\begin{figure}[ht]
\[
\begin{tikzpicture}[x=2em,y=-2em,baseline={([yshift=-2ex]current bounding box.center)}]
\Bsigma[0]{3}{2}
\Bsigma[1]{3}{1}
\URrungBL[2]{3}{2}{$k$}
\smnode at (0,-.3){$j-k$};
\smnode at (1,-.3){$i+k$};
\smnode at (2,-.3){$\ell$};
\smnode at (0,3.3){$\ell$};
\smnode at (1,3.3){$j$};
\smnode at (2,3.3){$i$};
\end{tikzpicture}
\qq{\longleftrightarrow}
\h^t
\begin{tikzpicture}[x=2em,y=-2em,baseline={([yshift=-2ex]current bounding box.center)}]
\URrungBL[0]{3}{1}{$k$}
\Bsigma[1]{3}{2}
\Bsigma[2]{3}{1}
\smnode at (0,-.3){$j-k$};
\smnode at (1,-.3){$i+k$};
\smnode at (2,-.3){$\ell$};
\smnode at (0,3.3){$\ell$};
\smnode at (1,3.3){$j$};
\smnode at (2,3.3){$i$};
\end{tikzpicture}
\]
\caption{Sliding a rung under a strand.  There are similar moves sliding any rung over or under any strand.  Any such move incurs a homological shift in the corresponding complex according to Proposition \ref{prop:isotopyShifts}.  There is also a $q$-degree shift which we will ignore.}
\label{fig:rungslide}
\end{figure}
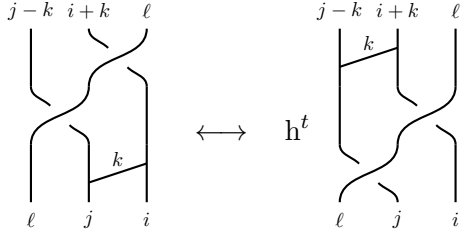

\begin{figure}[ht]
\[
\begin{tikzpicture}[x=2em,y=-2em,baseline={([yshift=-2ex]current bounding box.center)}]
\Bsigma[0]{2}{1}
\URrungBL[1]{2}{1}{$k$}
\smnode at (0,-.3){$j+k$};
\smnode at (1,-.3){$i-k$};
\smnode at (0,2.3){$i$};
\smnode at (1,2.3){$j$};
\end{tikzpicture}
\qq{\longleftrightarrow}
\h^t
\begin{tikzpicture}[x=2em,y=-2em,baseline={([yshift=-2ex]current bounding box.center)}]
\ULrungBR[0]{2}{1}{$k$}
\Bsigma[1]{2}{1}
\smnode at (0,-.3){$j+k$};
\smnode at (1,-.3){$i-k$};
\smnode at (0,2.3){$i$};
\smnode at (1,2.3){$j$};
\end{tikzpicture}
\]
\caption{Twisting a rung through a crossing.  There is a similar move for a rung slanted in the opposite direction as in the diagrams of Figure \ref{fig:braid as mcone ex}.  Any such move incurs a homological shift in the corresponding complex according to Proposition \ref{prop:isotopyShifts}.  There is also a $q$-degree shift which we will ignore.}
\label{fig:rungtwist}
\end{figure}
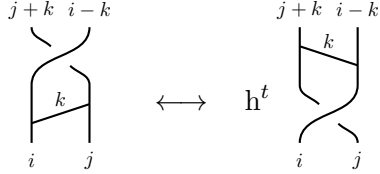

\begin{proposition}[\cite{MAMW} Proposition 2.19]\label{prop:isotopyShifts}
Let $T_1$ and $T_2$ be any two colored positive web-braid diagrams such that we can transform $T_1$ into $T_2$ via a sequence of braid-like Reidemeister moves, rung slides (see Figure \ref{fig:rungslide}) and rung twists (see Figure \ref{fig:rungtwist}).  Then $\KC{T_1}\simeq \h^t \q^s \KC{T_2}$ where
\begin{equation}\label{eq:isotopyShifts}
t = \sum_{\tau\in T_1} \min(\tau) - \sum_{\chi\in T_2} \min(\chi).
\end{equation}
Here each sum is taken over all crossings $\tau,\chi$ in each diagram.  The notation $\min(\tau)$ denotes the minimal color amongst the two strands at the crossing $\tau$.
\end{proposition}

If we look at Equation \eqref{eq:posxDef}, the homological size of the complex is dependent on the minimum of the two colors at hand, and so Equation \eqref{eq:isotopyShifts} is really just saying that, during such an isotopy of web-braids, the \emph{right-most} homological grading remains fixed regardless of the overall size of the complex.  In any case, if we want to use Proposition \ref{prop:isotopyShifts} to estimate homological shifts, we have to keep track of possible color-minimums of crossings for a given coloring.  This motivates the following definition.

\begin{definition}\label{def:colorsize}
Given a coloring $\gamma=(\gamma_1,\ldots,\gamma_n)$, we define the \emph{color size} of $\gamma$, denoted by $\cs{\gamma}$, to be
\begin{equation}\label{eq:colorsize}
\cs{\gamma}:=\sum_{1\leq i < j \leq n} \min(\gamma_i,\gamma_j).
\end{equation}
\end{definition}

We end this short section with a color size computation that we will require later.  Let $\Bn$ denote the braid group on $n$ strands and $\sigma_1,...,\sigma_{n-1}$ denote its standard (positive) generators.  We let $\FT$ denote the full twist braid $(\sigma_{1}\cdots\sigma_{n-1})^n$ on $n$ strands (see Figure \ref{fig:FT} for a specific example). 

\begin{figure}[ht]
\[
\begin{tikzpicture}[x=1em,y=-1em]
\foreach \j/\k in {0/1, 1/2, 2/3, 3/1, 4/2, 5/3, 6/1, 7/2, 8/3, 9/1, 10/2, 11/3}
 {
  \Bsigma[\j]{4}{\k}
 }
\end{tikzpicture}
\]
\caption{The full twist $\FT$ for $n=4$.}
\label{fig:FT}
\end{figure}

\begin{proposition}\label{prop:FT colorsize}
Given a coloring $\gamma$, let $\FT_{(\gamma)}$ denote the full twist on $n$ strands colored according to $\gamma$.  Then in the notation of Equation \eqref{eq:isotopyShifts}, we have
\begin{equation}\label{eq:FT colorsize}
\sum_{\chi\in \FT_{(\gamma)}} \min(\chi) = 2\cs{\gamma}
\end{equation}
\end{proposition}
\begin{proof}
A full twist involves each pair of strands crossing each other precisely twice.  Meanwhile, $\cs{\gamma}$ computes the sum of the minimum color between each pair of colors.
\end{proof}


\subsection{The complex for a uni-colored crossing}\label{sec:1-color crossing}
In order to prove Theorem \ref{thm:imprecise main thm}, we will need to understand the complexes associated to certain special colored braids.  The first and easiest scenario is the one utilized in \cite{MAMW}.  Consider the case of Equation \eqref{eq:posxDef} when $i=j$.  In other words, consider the complex associated to a uni-colored positive crossing.  In this case, the left-most term on the right-hand side of Equation \eqref{eq:posxDef} has a zero on both rungs, and so the diagram is in fact the identity diagram on two strands colored $i$.  We shall denote this 2-strand identity diagram by $I_{(i)}$.

\begin{lemma}\label{lem:1-color crossing complex}
Let $\tau_{(i)}$ denote a single positive crossing between two strands colored by $\gamma=(i,i)$.  Then $\KC{\tau_{(i)}}$ is equal to a cone
\[\KC{\tau_{(i)}}= \cone( I_{(i)} \rightarrow \Xch ) \]
where $\Xch$ is a complex such that, for any diagram $\delta_x\in\Xch$,
\begin{itemize}
\item $\h_X(\delta_x)\geq 0$, and
\item $\delta_x$ is a ladder diagram containing an intermediate coloring $\gamma_x$ with $\cs{\gamma_x}<\cs{\gamma}$.
\end{itemize}
\end{lemma}
\begin{proof} This is a direct translation of Equation \eqref{eq:posxDef} in the case when $i=j$; the complex $\Xch$ is precisely $\h^{-1}\KCnonzero{\tau_{(i)}}$ and the intermediate colorings are of the form $(i-\hkc(\delta),i+\hkc(\delta))$.  Visually we see
\[
\begin{tikzpicture}[x=2em,y=-2.5em,baseline={([yshift=-1ex]current bounding box.center)}]
\Bsigma[0]{2}{1}
\smnode at (0,-.3){$i$};
\smnode at (1,-.3){$i$};
\smnode at (0,1.3){$i$};
\smnode at (1,1.3){$i$};
\end{tikzpicture}
\qq{=}
\begin{tikzpicture}[x=2em,y=-2.5em,baseline={([yshift=-1ex]current bounding box.center)}]
\Bsigma[0]{2}{0}
\smnode at (0,-.3){$i$};
\smnode at (1,-.3){$i$};
\smnode at (0,1.3){$i$};
\smnode at (1,1.3){$i$};
\end{tikzpicture}
\qq{\longrightarrow}
\begin{tikzpicture}[x=2em,y=-1.5em,baseline={([yshift=-2ex]current bounding box.center)}]
\smnode at (0,-.3){$i$};
\smnode at (1,-.3){$i$};
\smnode at (0,2.3){$i$};
\smnode at (1,2.3){$i$};
\node[right,scale=.7] at (1,1){$i+1$};
\node[left,scale=.7] at (0,1){$i-1$};
\ULrungBR[0]{2}{1}{}
\URrungBL[1]{2}{1}{$1$}
\end{tikzpicture}
\qq{\longrightarrow}
\cdots
\qq{\longrightarrow}
\begin{tikzpicture}[x=2em,y=-1.5em,baseline={([yshift=-2ex]current bounding box.center)}]
\smnode at (0,-.3){$i$};
\smnode at (1,-.3){$i$};
\smnode at (0,2.3){$i$};
\smnode at (1,2.3){$i$};
\node[right,scale=.7] at (1,1){$2i$};
\node[left,scale=.7] at (0,1){$0$};
\ULrungBR[0]{2}{1}{}
\URrungBL[1]{2}{1}{$i$}
\end{tikzpicture}
\]
where intermediate colorings can be seen at the vertical midpoints of ladder diagrams.  As usual, we are ignoring the $q$-degree shifts.
\end{proof}

\subsection{The complex for a two-colored clasp}
The author learned the following argument from conversations with Matt Hogancamp.  We begin with a simple lemma from homological algebra.

\begin{lemma}\label{lem:Matts lemma}
Consider a chain complex $\Ach$ of the form
\[\Ach = A^0_1\oplus A^0_2 \rightarrow \Ach^{\geq 1}\]
over some additive category, with no terms in negative homological grading.  Let $\pi^*:\Ach \rightarrow A^0_2$ be the projection map onto the one-term complex sitting in homological degree zero.  Suppose that $\pi^*$ is chain-homotopic to the zero map.  Then the complex $\Ach$ is chain homotopy equivalent to one of the form
\[\Ach \simeq A^0_1 \rightarrow \Xch\]
where the term $A^0_2$ has been removed, and the complex $\Xch$ is a direct summand of the original $\Ach^{\geq 1}$.
\end{lemma}
\begin{proof}
In this setting, a chain homotopy between $\pi^*$ and $0$ is precisely a map $H:A^1\rightarrow A^0_2$ that inverts the corresponding component of the differential, which in turn allows for a Gaussian elimination argument (perhaps after passing to the Karoubi envelope of our category).  The details are left to the reader.
\end{proof}

In the statement of the following lemma, we use the notation $I_{(ij)}$ to indicate the two-strand identity diagram with strands colored $i$ and $j$.
\begin{lemma}\label{lem:2-color clasp complex}
Let $\tau^2_{(ij)}$ denote a pair of adjacent positive crossings (called a \emph{clasp}) between two strands colored by $\gamma=(i,j)$.  Then $\KC{\tau^2_{(ij)}}$ is chain homotopy equivalent to a cone
\[\KC{\tau^2_{(ij)}} \simeq \cone(I_{(ij)}\rightarrow \Xch)\]
where $\Xch$ is a direct summand of a complex $\Cch$ satisfying, for any diagram $\delta_x\in\Cch$,
\begin{itemize}
\item $\h_C(\delta_x)\geq 0$, and
\item $\delta_x$ is a ladder diagram containing an intermediate coloring $\gamma_x$ with $\cs{\gamma_x}<\cs{\gamma}$.
\end{itemize}
\end{lemma}
\begin{proof}
We will consider the case where $j\geq i$ (the case of $j<i$ is similar).  We begin by expanding $\KC{\tau^2_{(ij)}}$ along both crossings into a complex of various web diagrams.  It is not hard to see from Equation \eqref{eq:posxDef} that the only term in homological degree zero is the trapezoid
\[\begin{tikzpicture}[x=2em,y=-2em,baseline={([yshift=-2ex]current bounding box.center)}]
\URrungBL[0]{2}{1}{$j-i$}
\ULrungBR[1]{2}{1}{$j-i$}
\smnode at (0,-.3){$i$};
\smnode at (1,-.3){$j$};
\smnode at (0,2.3){$i$};
\smnode at (1,2.3){$j$};
\end{tikzpicture}
\]
while every other diagram, sitting in $\KCnonzero{\tau^2_{(ij)}}$, contains intermediate colorings of lesser color-size.  We will declare $\Cch$ to be the complex $\KCnonzero{\tau^2_{(ij)}}$.  If $j=i$, we are done (letting $\Xch=\Cch$, trivially a direct summand as desired).  Otherwise, we use the 2-isomorphism which lifts the second equation of Figure \ref{fig:webRel} to replace this trapezoid with a large direct sum of terms (all still in homological degree zero).  A careful inspection of the right-hand side of Figure \ref{fig:webRel} indicates that precisely one of these summands will be the identity diagram $I_{(ij)}$, while the other summands will contain ($q$-shifted) trapezoids of the form
\[
\begin{tikzpicture}[x=2em,y=-2em,baseline={([yshift=-2ex]current bounding box.center)}]
\ULrungBR[0]{2}{1}{$p$}
\URrungBL[1]{2}{1}{$p$}
\smnode at (0,-.3){$i$};
\smnode at (1,-.3){$j$};
\smnode at (0,2.3){$i$};
\smnode at (1,2.3){$j$};
\end{tikzpicture}
\]
for $p>0$.

Now we fix any such trapezoidal diagram, calling it $\epsilon_p$.  As in Lemma \ref{lem:Matts lemma}, we consider the projection map
\[\pi^*:\KC{\tau^2_{(ij)}} \rightarrow \KC{\epsilon_p}\]
to the corresponding one-term complex in homological degree zero.  Using the duality equivalence of Figure \ref{fig:duality iso}, together with the fact that any such trapezoid $\epsilon_p$ is symmetric about its central horizontal axis, we deduce the following:
\[
\pi^*\in \hom\left(
\,
\begin{tikzpicture}[x=1.5em,y=-1.5em,baseline={([yshift=-2ex]current bounding box.center)}]
\Bsigma[0]{2}{1}
\Bsigma[1]{2}{1}
\smnode at (0,-.3){$i$};
\smnode at (1,-.3){$j$};
\smnode at (0,2.3){$i$};
\smnode at (1,2.3){$j$};
\end{tikzpicture}
\qq{,}
\begin{tikzpicture}[x=1.5em,y=-1.5em,baseline={([yshift=-2ex]current bounding box.center)}]
\ULrungBR[0]{2}{1}{$p$}
\URrungBL[1]{2}{1}{$p$}
\smnode at (0,-.3){$i$};
\smnode at (1,-.3){$j$};
\smnode at (0,2.3){$i$};
\smnode at (1,2.3){$j$};
\end{tikzpicture}
\,
\right)
\,\simeq\,
\hom\left(
\,
\begin{tikzpicture}[x=1.5em,y=-1.5em,baseline={([yshift=-2ex]current bounding box.center)}]
\Bsigma[0]{2}{1}
\Bsigma[1]{2}{1}
\ULrungBR[2]{2}{1}{$p$}
\URrungBL[3]{2}{1}{$p$}
\smnode at (0,-.3){$i$};
\smnode at (1,-.3){$j$};
\smnode at (0,4.3){$i$};
\smnode at (1,4.3){$j$};
\end{tikzpicture}
\qq{,}
\begin{tikzpicture}[x=1.5em,y=-1.5em,baseline={([yshift=-2ex]current bounding box.center)}]
\foreach \j in {0,1,2,3}
 { \Bsigma[\j]{2}{0} }
\smnode at (0,-.3){$i$};
\smnode at (1,-.3){$j$};
\smnode at (0,4.3){$i$};
\smnode at (1,4.3){$j$};
\end{tikzpicture}
\,
\right)
\,\simeq\,
\hom\left(
\,
\h^{2p}
\begin{tikzpicture}[x=1.5em,y=-1.5em,baseline={([yshift=-2ex]current bounding box.center)}]
\ULrungBR[0]{2}{1}{$p$}
\Bsigma[1]{2}{1}
\Bsigma[2]{2}{1}
\URrungBL[3]{2}{1}{$p$}
\smnode at (0,-.3){$i$};
\smnode at (1,-.3){$j$};
\smnode at (0,4.3){$i$};
\smnode at (1,4.3){$j$};
\end{tikzpicture}
\qq{,}
\begin{tikzpicture}[x=1.5em,y=-1.5em,baseline={([yshift=-2ex]current bounding box.center)}]
\foreach \j in {0,1,2,3}
 { \Bsigma[\j]{2}{0} }
\smnode at (0,-.3){$i$};
\smnode at (1,-.3){$j$};
\smnode at (0,4.3){$i$};
\smnode at (1,4.3){$j$};
\end{tikzpicture}
\,
\right).
\]
The positive homological shift in the first diagram of the last hom-space is derived from Proposition \ref{prop:isotopyShifts}.  This shift ensures that this hom-space is actually the zero-space, and thus we must have $\pi^*$ chain homotopic to the zero map.  Lemma \ref{lem:Matts lemma} allows us to conclude that there is a Gaussian elimination argument which removes any such diagram $\epsilon_p$ from the complex $\KC{\tau^2_{(ij)}}$, replacing $\Cch$ with some direct summand of itself.  Iterating this process to remove all such trapezoids from the homological degree zero term leaves us with a complex of the desired form.
\end{proof}


\section{Proving Theorem \ref{thm:imprecise main thm}}\label{sec:main pf}
Theorem \ref{thm:imprecise main thm} is about limiting complexes associated with certain semi-infinite braids.  Such braids will be represented by semi-infinite braid words up to finitely many Reidemeister moves.  In order to prove Theorem \ref{thm:imprecise main thm}, we begin by proving a version for semi-infinite braid words; this will take up the majority of this section.  Once we have this result (Section \ref{sec:proving main thm}), it will be relatively simple to lift the result to semi-infinite braids.

The overall strategy for the semi-infinite braid words is similar to that of \cite{MWGI,MAMW}, and we summarize it here.  We will construct inverse systems for the infinite twist, as well as for other semi-infinite words, where the system maps are quotient maps corresponding to resolving certain crossings as vertical identity braids (Proposition \ref{prop:colorpure inv sys}).  In the case of the infinite twist, it will be relatively simple to show that the system is Cauchy and thus has a limit categorifying a projector as in Cautis' results (Theorem \ref{thm:infFT with max purity seq}).  For the systems coming from other semi-infinite words, we will construct maps $F_\ell$ to the system for the infinite twist that are \emph{also} quotient maps corresponding to identity resolutions (Section \ref{sec:constructing Fell}); this will guarantee that these maps commute with all system maps.  We will then seek to apply Proposition \ref{prop:compSys} by estimating the homological orders $|F_\ell|_\h$ of these maps; this will be done by expanding $\cone(F_\ell)$ as a multicone, simplifying the multicone using Proposition \ref{prop:gen multicone equiv}, and estimating the resulting homological orders of elements using Propositions \ref{prop:isotopyShifts} and \ref{prop:FT colorsize} (Section \ref{sec:bounding Fell}).

Throughout this section, $I$ will denote the identity braid on $n$ strands.

\subsection{Color-pure braids}\label{sec:colorpure braids}
If $\beta\in\Bn$ is a braid on $n$ strands, we will use the notation $\beta^{(\gamma)}_{(\gamma')}$ to indicate that $\beta$ has been colored so that the sequence of colors at the top (respectively bottom) of $\beta$ is given by $\gamma$ (respectively $\gamma'$).  Of course, since $\beta$ is a braid, $\gamma'$ must be some permutation of $\gamma$ determined by $\beta$.

\begin{definition}\label{def:colorpure braids}
A colored braid $\beta^{(\gamma)}_{(\gamma')}$ is called \emph{color-pure (with respect to $\gamma$)} if $\gamma=\gamma'$.  In such a case we will omit the superscript and simply write $\beta_{(\gamma)}$ for the color-pure braid.
\end{definition}

Clearly all pure braids are color-pure with respect to any $\gamma$.  If $\gamma=(i,i,\ldots,i)$ denotes a uni-coloring, then all colored braids $\beta^{(\gamma)}_{(\gamma')}$ are color-pure.  Meanwhile, if $\gamma=(i,j,\ldots,\ell)$ denotes a coloring using all distinct colors, then a colored braid $\beta^{(\gamma)}_{(\gamma')}$ is color-pure if and only if it is pure.  Note that the full twist $\FT$ is pure, and so our earlier notation $\FT_{(\gamma)}$ agrees with the convention of omitting the superscript for Definition \ref{def:colorpure braids}.

Our first task is to combine Lemmas \ref{lem:1-color crossing complex} and \ref{lem:2-color clasp complex} into a more general statement about color-pure braids.  We begin with a helpful lemma that will help us find clasps.
\begin{lemma}\label{lem:clasps in colorpure}
Any non-trivial positive color-pure braid $\beta_{(\gamma)}$ that does not have any uni-colored crossings must be braid isotopic to another positive color-pure braid $\beta'_{(\gamma)}$ that contains some clasp $\tau^2$.
\end{lemma}
\begin{proof}
We begin by noting that a color-pure braid $\beta$ that does not have any uni-colored crossings must in fact be pure.  Indeed if we have a color-pure braid that is not pure, there must be some strand that begins at one colored end-point $a$ and ends at a different end-point $b$ of the same color.  If this is the case, then the strand that begins at $b$ must end at another end-point $c$ of the same color without crossing the first strand.  It is not hard to see that this must eventually lead to a contradiction due to having only finitely many strands.

Therefore, if we let $\Phi$ denote the usual map from the braid group $\Bn$ to the symmetric group $S_n$, we must have $\Phi(\beta)=I$, the identity permutation.  It is well known that $\Bn$ and $S_n$ share the same generators (transpositions) and relations, with $S_n$ having the additional relation $\sigma_i^2=I$ for any generator.  Since $\beta$ was positive and non-trivial, but $\Phi(\beta)=I$, we can conclude that $\beta$ can be altered via braid relations into some isotopic (positive) braid $\beta'$ that contains a clasp $\tau^2$ as desired.
\end{proof}

We now state and prove the key proposition about the structure of the complex associated to a positive color-pure braid.

\begin{proposition}\label{prop:colorpure complex}
Given a positive color-pure braid $\beta_{(\gamma)}$ on $n$ strands, the complex $\KC{\beta_{(\gamma)}}$ is chain homotopy equivalent to a cone
\begin{equation}\label{eq:colorpure complex}
\KC{\beta_{(\gamma)}} \simeq \cone(I_{(\gamma)} \rightarrow \Xch)
\end{equation}
where $\Xch$ is a direct summand of a complex $\Cch$ satisfying, for any diagram $\delta_x\in\Cch$,
\begin{itemize}
\item $\h_C(\delta_x)\geq 0$, and
\item $\delta_x$ is a ladder diagram containing an intermediate coloring $\gamma_x$ with $\cs{\gamma_x}<\cs{\gamma}$.
\end{itemize}
The simplified complex written in this way will be denoted $\KCsimp{\beta_{(\gamma)}}$.
\end{proposition}
\begin{proof}
Let $\beta_1:=\beta_{(\gamma)}$ denote our starting braid.  If there is a uni-colored crossing $\sigma_a$ in $\beta_1$, we apply Lemma \ref{lem:1-color crossing complex} to $\sigma_a$ allowing us to write $\KC{\beta_1}$ as a cone as in  a two-term version of Figure \ref{fig:braid as mcone ex}:
\[\KC{\beta_1} = \cone \left( \KC{\beta_2} \rightarrow \Xch_1 \right).\]
Here the braid $\beta_2$ is derived from $\beta_1$ by deleting $\sigma_a$, while the complex $\Xch_1$ is based upon concatenating $\h^{-1}\KCnonzero{\sigma_a}$ with the rest of the braid.  For the sake of notational convenience moving forward, we declare $\Cch_1:=\Xch_1$ so that $\Xch_1$ is trivially a direct summand of the complex $\Cch_1$.

Now since we have not changed any of the colors at the top or bottom of the braid, we see that $\beta_2$ is again color-pure but with one less crossing than $\beta_1$.  If $\beta_2$ contains another uni-colored crossing, we apply the same reasoning again to get
\[\KC{\beta_1} = \cone\big( \cone \left( \KC{\beta_3} \rightarrow \Xch_2 \right) \rightarrow \Xch_1 \big).\]
We can continue in this way until we have written $\KC{\beta_1}$ as a large iterated cone with initial term $\KC{\beta_\ell}$ where $\beta_\ell$ is a positive color-pure braid with \emph{no} uni-colored crossings (here $\ell-1$ is the number of uni-colored crossings that were in $\beta_1$, which have now been deleted).  Each of the complexes $\Xch_k$ satisfy the properties of Lemma \ref{lem:1-color crossing complex}, and we can view any $\Xch_k$ as a trivial direct summand of $\Cch_k:=\Xch_k$ satisfying these same properties.

According to Lemma \ref{lem:clasps in colorpure}, our color-pure $\beta_\ell$ with no uni-colored crossings must be braid isotopic to some $\beta_\ell'$ having a clasp $\tau^2$.  Thus we can apply Lemma \ref{lem:2-color clasp complex} to this clasp and write $\KC{\beta_\ell}$ as a cone again as in Figure \ref{fig:braid as mcone ex}:
\[\KC{\beta_\ell}\simeq \KC{\beta_\ell'} \simeq \cone\left( \KC{\beta_{\ell+1}} \rightarrow \Xch_\ell \right) \]
where the braid $\beta_{\ell+1}$ is obtained from $\beta_\ell'$ by deleting the clasp, while $\Xch_\ell$ is a direct summand of the complex $\Cch_\ell$ which is based upon concatenating $\h^{-1}\KCnonzero{\tau^2}$ with the rest of the braid as in the logic of Figure \ref{fig:braid as mcone ex}.  Since a cone is just a two-term multicone, Proposition \ref{prop:gen multicone equiv} allows us to plug this cone in place of the complex $\KC{\beta_\ell}$ in the iterated cone for $\KC{\beta_1}$, so we have
\[
\KC{\beta_1} \simeq \cone\bigg(\cone\Big(\hspsymbol{.1cm}{\cdots} \cone( \KC{\beta_{\ell+1}} \rightarrow \Xch_\ell ) \hspsymbol{.1cm}{\cdots} \rightarrow \Xch_2 \Big) \rightarrow \Xch_1 \bigg).
\]

Thus $\beta_{\ell+1}$ is again a color-pure braid with no uni-colored crossings, and so we can apply the same reasoning yet again.  We continue in this way until we have $\KC{\beta_1}$ written as a large iterated cone with initial term $I_{(\gamma)}$, the one-term complex associated to the identity braid colored by $\gamma$:
\[
\KC{\beta_1} \simeq \cone \bigg( \cone \Big( \cdots \cone \left( \cone( I_{(\gamma)} \rightarrow \Xch_p) \rightarrow \Xch_{p-1}\right) \cdots \rightarrow \Xch_2 \Big) \rightarrow \Xch_1 \bigg)
\]
Here $p$ describes the number of steps it took to delete all of the uni-colored crossings and two-colored clasps from $\beta_1$ to arrive at the identity braid.  Every complex $\Xch_k$ is a direct summand of some $\Cch_k$ whose terms satisfy the desired properties.  The cone operation simply takes direct sums of the objects in the complexes (with a positive homological shift), and so the statement of the proposition is satisfied by taking $\Xch:=\bigoplus_{k=1}^p \Xch_k$ and $\Cch:=\bigoplus_{k=1}^p \Cch_k$.


\end{proof}

\subsection{Color-pure semi-infinite braids}
We begin by giving a more precise definition for our main class of braids.  Let $\BnGen$ denote the standard set of multiplicative generators for the braid group $\Bn$, with $\BnGen^+\subset\BnGen$ denoting the subset consisting of only positive generators $\sigma_i$.

\begin{definition} \label{def:semi-inf braid word}
A \emph{semi-infinite braid word} $\B$ on $n$ strands is a map $\B:\N\rightarrow\BnGen$, written as an infinite word
\[\B = \sigma_{i_1}^{\epsilon_1}\sigma_{i_2}^{\epsilon_2}\sigma_{i_3}^{\epsilon_3}\cdots\]
on the generators $\sigma_i$ (where each $\epsilon_i\in\{-1,1\}$).  We call $\B$ \emph{positive} if $\im(\B)\subset\BnGen^+$ (equivalently each $\epsilon_i=1$, so we may ignore them from the notation).  
\end{definition}

In order to use the results of Section \ref{sec:InvSys}, we need to describe a semi-infinite braid word as a limit of finite braid words.  First we establish notation.

\begin{definition} \label{def:partial braid}
Let $\B = \sigma_{i_1}^{\epsilon_1}\sigma_{i_2}^{\epsilon_2}\sigma_{i_3}^{\epsilon_3}\cdots$ be a semi-infinite braid word and $\ell \in \N$. We define the \emph{$\ell$th partial braid} of $\B$, denoted by $\B_\ell$, as
\[\B_\ell = \sigma_{i_1}^{\epsilon_1}\sigma_{i_2}^{\epsilon_2}\cdots\sigma_{i_\ell}^{\epsilon_\ell}.\]
More generally, the \emph{partial sub-braid from $a$ to $b$}, denoted by $\B^a_b$, is defined to be the braid word 
\[\B^a_b:=\sigma_{i_{a+1}}^{\epsilon_{i_{a+1}}}\sigma_{i_{a+2}}^{\epsilon_{i_{a+2}}}\cdots\sigma_{i_b}^{\epsilon_{i_b}}.\]
In this way, the $\ell^{\text{th}}$ partial braid $\B_\ell$ is equivalent to $\B^0_\ell$, and in general $\B^a_b \cdot \B^b_c = \B^a_c$.  We will also use the notation $\B^a_\infty$ to denote the \emph{truncated} infinite braid word obtained from $\B$ by deleting the first $a$ crossings.
\end{definition}

In this text we are concerned with colored braids.  In our previous papers on the subject, when all strands in a braid were colored by $i$, the uni-coloring $(i,i,\dots,i)$ at the start and end of a semi-infinite braid word $\B$ could easily be identified with the coloring at the start and end of any partial braid $\B_\ell$.  When we allow arbitrary colorings however, more care is needed.

\begin{definition}\label{def:colorpure semi-inf braid word}
Let $\gamma=(\gamma_1,\dots,\gamma_n)\in\{0,1,\dots,N\}^n$ be a fixed coloring.  Then a semi-infinite braid word $\B:\N\rightarrow\BnGen$ gives rise to a sequence of colorings $\gamma(\cdot):\N\cup\{0\} \rightarrow \{0,1,\dots,N\}^n$ by defining $\gamma(0):=\gamma$, and then defining $\gamma(\ell)$ for $\ell\in\N$ as the coloring at the bottom of $\B_\ell$ determined by coloring the top by $\gamma$.  We call $\B$ \emph{color-pure with respect to $\gamma$} if $\gamma(\ell)=\gamma$ for infinitely many $\ell\in\N$, and denote it $\B_{(\gamma)}$.  The sequence $m_1 < m_2 < \cdots$ of all $m_i$ such that $\gamma(m_i)=\gamma$ is called the \emph{maximal purity sequence} for $\B_{(\gamma)}$.
\end{definition}
It is clear that, after fixed an `initial' coloring $\gamma$ for the `top' of $\B$, any partial sub-braid $\B^a_b$ within $\B$ is colored as $(\B^a_b)^{(\gamma(a))}_{(\gamma(b))}$.

With the notation set up, we have the following proposition allowing us to build complexes for certain semi-infinite braid words.

\begin{proposition}\label{prop:colorpure inv sys}
If $\B_{(\gamma)}$ is a positive semi-infinite color-pure braid word with maximal purity sequence $m_1<m_2<\cdots$, then there is a corresponding inverse system of complexes
\[\{\KC{\B_{m_k}},f_k\}=\KC{\B_{m_0}} \xleftarrow{f_0} \KC{\B_{m_1}} \xleftarrow{f_1} \KC{\B_{m_2}} \xleftarrow{f_2} \cdots\]
where $\B_{m_0}$ is defined to be the colored identity braid $I_{(\gamma)}$ and the maps $f_k$ are the quotient maps implied by Proposition \ref{prop:colorpure complex}.  Furthermore, if this inverse system is Cauchy with limit denoted $\KC{\B_\infty}$ and $\ell_1<\ell_2<\cdots$ is any subsequence of the maximal purity sequence, then there is a corresponding inverse system
\[\{\KC{\B_{\ell_j}},g_j\}=\KC{\B_{\ell_0}} \xleftarrow{g_0} \KC{\B_{\ell_1}} \xleftarrow{g_1} \KC{\B_{\ell_2}} \xleftarrow{g_2} \cdots\]
which is also Cauchy with a chain homotopy equivalent limit
\[\lim \KC{\B_{\ell_j}} \simeq \KC{\B_\infty}.\]
\end{proposition}
\begin{proof}
First we describe the maps $f_k$ in slightly more detail.  Fixing $k$, we let $\beta$ denote the partial sub-braid $\B^{m_k}_{m_{k+1}}$ of crossings in $\B_{m_{k+1}}$ that are not in $\B_{m_k}$; by assumption $\beta$ is color-pure with respect to $\gamma$.  If we use Proposition \ref{prop:colorpure complex} to write the simplified complex $\KCsimp{\beta}$ for $\beta$, and then expand $\KC{\B_{m_{k+1}}}$ along this cone as in Figure \ref{fig:braid as mcone ex}, we see (with a slight abuse of notation)
\[
\KC{\B_{m_{k+1}}}
\qq{\simeq}
\cone\left(
\begin{tikzpicture}[x=2em,y=-2em,baseline={([yshift=-2ex]current bounding box.center)}]
\Bbox[0]{2}{1}{2}{$\B_{m_k}$}
\Bbox[1]{2}{1}{2}{$I$}
\end{tikzpicture}
\qq{\longrightarrow}
\begin{tikzpicture}[x=2em,y=-2em,baseline={([yshift=-2ex]current bounding box.center)}]
\Bbox[0]{2}{1}{2}{$\B_{m_k}$}
\Bbox[1]{2}{1}{2}{$\Xch$}
\end{tikzpicture}
\right).
\]

The map $f_k$ is the quotient map from this cone to the complex $\KC{\B_{m_k}}$.  From this it is clear that a subsequence $\ell_1<\ell_2<\cdots$ has its own set of quotient maps $g_j$ and we can construct a commuting diagram of inverse systems as in Figure \ref{fig:compare inv sys} where all of the horizontal maps are identity maps.  Being a subsequence ensures that our new inverse system is also Cauchy (details here are left to the reader), and since identity maps have infinite homological order, Proposition \ref{prop:compSys} gives us our desired chain homotopy equivalence.
\end{proof}

We can now present the main definitions for semi-infinite braids.

\begin{definition}\label{def:semi-inf braid}
A \emph{semi-infinite braid} $\BB$ is an equivalence class of semi-infinite braid words $\B$, where $\B$ is equivalent to $\B'$ if and only if $\B'$ can be arrived at from $\B$ via a finite set of braid moves.  We call $\BB$ \emph{positive} if some choice of representative word is positive.  We call $\BB$ \emph{color-pure with respect to $\gamma$}, and denote it $\BB_{(\gamma)}$, if any (and thus every) word representing $\BB$ is color-pure (note that this condition is not affected by finitely many braid moves).
\end{definition}

\begin{remark}\label{rmk:fin many braid moves only}
The restriction to allowing only finitely many braid moves was neglected in our earlier papers \cite{MWGI,MAMW}, but it is clearly necessary to avoid situations where a sequence of braid moves starting from the word $\B$ `limits' to a new word $\B'$ that does not share the properties of $\B$.  For instance consider the following sequence of braid moves in $\Bn$ for $n=4$:
\begin{align*}\B &= \sigma_3\sigma_1\sigma_1\sigma_1\cdots\\
& \cong \sigma_1\sigma_3\sigma_1\sigma_1\cdots\\
& \cong \sigma_1\sigma_1\sigma_3\sigma_1\sigma_1\cdots\\
& \quad\vdots
\end{align*}
which would show `in the limit' that $\B$ is equivalent to the infinite twist on the first two strands with \emph{no} occurrence of $\sigma_3$.  We wish to disallow such limiting statements.
\end{remark}

\begin{proposition}\label{prop:colorpure cx well defined}
Let $\BB_{(\gamma)}$ be a positive color-pure semi-infinite braid, and let $\B$ and $\B'$ be two positive semi-infinite words representing $\BB$ with corresponding inverse systems $\{\KC{\B_{m_k}},f_k\}$ and $\left\{\KC{\B'_{m'_\ell}},f'_\ell\right\}$ via Proposition \ref{prop:colorpure inv sys}.  Then there are maps $F_\ell:\KC{\B'_{m'_\ell}}\rightarrow\KC{\B_{m_k}}$ commuting with the system maps such that $|F_\ell|_\h\rightarrow\infty$ as $\ell\rightarrow\infty$.  In particular, if $\{\KC{\B_{m_k}},f_k\}$ is Cauchy, then so is $\left\{\KC{\B'_{m'_\ell}},f'_\ell\right\}$ and their limits are chain homotopy equivalent.
\end{proposition}
\begin{proof}
If $\B$ and $\B'$ are related by finitely many Reidemeister moves, then there exists some $\kappa$ and $\lambda$ such that the partial braid words $\B_{m_\kappa}$ and $\B'_{m'_\lambda}$ are braid isotopic, and the truncated semi-infinite braid words $\B^{m_\kappa}_\infty$ and $\B'^{m'_\lambda}_\infty$ are equal.
\[\B_{m_\kappa} \cong \B'_{m'_\lambda} \qq{,}  \B^{m_\kappa}_\infty = \B'^{m'_\lambda}_\infty\]
The maps $F_\ell$, for $\ell>\lambda$, are the chain homotopy equivalences induced by this braid isotopy while keeping the `later' crossings $\B'^{m'_\lambda}_{m_\ell}$ fixed.  Since the system maps of Proposition \ref{prop:colorpure inv sys} beyond this point rely only on resolutions of color-pure braids within $\B'^{m'_\lambda}_{m'_\ell}$ and $\B^{m_\kappa}_{m_k}$, our maps $F_\ell$ trivially commute with the system maps.  

Furthermore, because each of these $F_\ell$ are chain homotopy equivalences, we have $|F_\ell|_\h = \infty$ and $f'_\ell \sim F_{\ell-1}^{-1} f_k F_\ell$ so that $\left\{\KC{\B'_{m'_\ell}},f'_\ell\right\}$ is Cauchy and we are done via Proposition \ref{prop:compSys}.  (The maps $F_\ell$ for $\ell\leq \lambda$ are irrelevant and can be taken to be projection maps to the identity $\B_0$ via Proposition \ref{prop:colorpure complex}, trivially commuting with all system maps.)  
\end{proof}

Thus for a positive semi-infinite color-pure braid $\BB_{(\gamma)}$, we have a well-defined inverse system up to maps $F_\ell$ connecting any two such systems arising from different (positive) words representing $\BB_{(\gamma)}$.  If any such word provides a Cauchy inverse system, we have a well-defined limiting system $\KC{\BB_{(\gamma)}}$ up to chain homotopy equivalence.  The requirement that we consider only positive representative words is only a crutch for the moment; we will see in Section \ref{sec:neg crossings} that allowing words with negative crossings produces degree shifts (as is to be expected; we've already remarked upon this fact for Reidemeister II moves in our normalization) but does not change the nature of our results.

\subsection{The infinite full twist}\label{sec:Inf Twist}
Ultimately, we will use Proposition \ref{prop:compSys} to compare our Cauchy inverse systems to a well understood Cauchy inverse system coming from studying the infinite full twist which we now discuss.

Fix $n\in\N$.  We begin by noting that a single full twist $\FT$ is a positive pure braid (and thus is color pure with respect to any coloring $\gamma$), so that the semi-infinite braid word $\FT^\infty:=\FT\FT\FT\cdots$ is color-pure with respect to any $\gamma$ via the sequence of partial braid words defined by $(\FT^\infty)_{m_k}:=\FT^k$.  We can therefore speak of the colored infinite full twist $\FT_{(\gamma)}^\infty$.  In \cite{RozInf}, Rozansky proved that the standard (uncolored, $\mathfrak{sl}_2$) Khovanov complex of the infinite full twist was a well-defined complex and that it categorified the Jones-Wenzl projector.  We now state the analogous result in colored Khovanov-Rozansky homology due to Cautis.

\begin{theorem}[Cautis \cite{Cau12}]\label{thm:infFT}
Let $\FT_{(\gamma)} \in \Bn$ denote the positive full twist braid on $n$ strands colored by $\gamma$.  Then, taking all complexes over the ground ring $\mathbb{C}$,
\[\lim_{k \to \infty} \KC{\FT^k_{(\gamma)}}\]
is a well-defined idempotent chain complex categorifying a highest weight projector 
\[p_{n,(\gamma)}: \bigotimes_{i=1}^n\left(\bigwedge^{\gamma_i} (\CC_q^N)\right) \rightarrow \bigotimes_{i=1}^n\left(\bigwedge^{\gamma_i} (\CC_q^N)\right)\]
factoring through the unique highest weight subrepresentation of $\otimes_{i=1}^n(\bigwedge^{\gamma_i} (\CC_q^N))$. 
\end{theorem}

In this text, we will write $P_{n,(\gamma)}:= \lim_{k \to \infty} \KC{\FT^k_{(\gamma)}}$ for Cautis' limiting complex for the sake of brevity.

\begin{remark} \label{rmk-Rickard}
Strictly speaking, Theorem \ref{thm:infFT} is proven for Rickard complexes in the categorification of $\qsln$. However this category is equivalent to $\foamC$ and thus can be viewed as a result for complexes associated to full twist braids in $\foamC$ (See \cite{Cau12} and \cite{QR} for more details).
\end{remark} 

As indicated above, Cautis' proof assumes that the ground ring is actually the field $\mathbb{C}$.  Furthermore, Cautis' complex is based on taking the inverse system built from the sequence of complete full twists, which may not necessarily be the maximal purity sequence for the semi-infinite colored braid $\FT_{(\gamma)}^\infty$.  In the theorem below we will prove the existence of this limiting complex over any ground ring using the maximal purity sequence for the given coloring $\gamma$, and then Proposition \ref{prop:colorpure inv sys} will show that our maximal purity sequence version recovers $P_{n,(\gamma)}$ in the case that the ground ring was $\mathbb{C}$.  The proof of this theorem provides a good warm-up for the proof of Theorem \ref{thm:imprecise main thm} to come afterwards.

\begin{theorem}\label{thm:infFT with max purity seq}
Fix a coloring $\gamma$ and let $p_1<p_2<\cdots$ denote the maximal purity sequence for the colored infinite full twist $\infFT:=\FT_{(\gamma)}^\infty$.  Then the corresponding inverse system $\{\KC{\infFT_{p_k}},f_k\}$ is Cauchy with limiting complex denoted
\[ P_{n,(\gamma)}:=\KC{\infFT_\infty}.\]
In particular, when the ground ring is $\mathbb{C}$, this limiting complex is chain homotopy equivalent to the complex of Cautis in Theorem \ref{thm:infFT}, justifying the notation.
\end{theorem}
\begin{proof}
By definition, we must show that the homological orders of the maps $|f_k|_\h$ grow infinite as $k\rightarrow\infty$.

Fix $k>0$.  As in the proof of Proposition \ref{prop:colorpure inv sys}, we apply Proposition \ref{prop:colorpure complex} to the (color-pure) partial sub-braid $\infFT^{p_k}_{p_{k+1}}$.  Then $f_k$ is defined as a quotient map on the corresponding cone (we will omit the notations $\KC{\cdot}$ as usual in the visual presentations):
\[
\cone\left(
\begin{tikzpicture}[x=2em,y=-2em,baseline={([yshift=-2ex]current bounding box.center)}]
\Bbox[0]{2}{1}{2}{$\infFT_{p_k}$}
\Bbox[1]{2}{1}{2}{$I$}
\end{tikzpicture}
\qq{\longrightarrow}
\begin{tikzpicture}[x=2em,y=-2em,baseline={([yshift=-2ex]current bounding box.center)}]
\Bbox[0]{2}{1}{2}{$\infFT_{p_k}$}
\Bbox[1]{2}{1}{2}{$\Xch$}
\end{tikzpicture}
\right) \qq{\xrightarrow{f_k}}
\begin{tikzpicture}[x=2em,y=-2em,baseline={([yshift=-2ex]current bounding box.center)}]
\Bbox[0]{2}{1}{2}{$\infFT_{p_k}$}
\Bbox[1]{2}{1}{2}{$I$}
\end{tikzpicture}
\qq{=} \begin{tikzpicture}[x=2em,y=-2em,baseline={([yshift=-2ex]current bounding box.center)}]
\Bbox[0]{2}{1}{2}{$\infFT_{p_k}$}
\end{tikzpicture}.
\]
As such, we have $\cone(f_k)\simeq \KC{ \infFT_{p_k} }\otimes \Xch$
and since $\Xch$ is a direct summand of a complex $\Cch$ whose terms satisfy the properties of Proposition \ref{prop:colorpure complex}, $\cone(f_k)$ is likewise chain homotopy equivalent to a direct summand of the complex $\KC{\infFT_{p_k}} \otimes \Cch$.

We now expand the complex $\KC{\infFT_{p_k}} \otimes \Cch$ as a multicone along $\Cch$ as expressed visually below:
\[
\begin{tikzpicture}[x=2em,y=-2em,baseline={([yshift=-2ex]current bounding box.center)}]
\Bbox[0]{2}{1}{2}{$\infFT_{p_k}$}
\Bbox[1]{2}{1}{2}{$\Cch$}
\end{tikzpicture}
 \qq{\simeq}
\underset{\delta_x,\delta_y\in\Cch}{\Mcone}\left( 
\begin{tikzpicture}[x=2em,y=-2em,baseline={([yshift=-2ex]current bounding box.center)}]
\Bbox[0]{2}{1}{2}{$\infFT_{p_k}$}
\Bbox[1]{2}{1}{2}{$\delta_x$}
\end{tikzpicture}
\qq{\longrightarrow}
\begin{tikzpicture}[x=2em,y=-2em,baseline={([yshift=-2ex]current bounding box.center)}]
\Bbox[0]{2}{1}{2}{$\infFT_{p_k}$}
\Bbox[1]{2}{1}{2}{$\delta_y$}
\end{tikzpicture}
\right).
\]

From here, we consider terms $\infFT_{p_k} \cdot \delta_x$ within this multicone.  First of all, since $\infFT=\FT\FT\cdots$ we must have
\[
\begin{tikzpicture}[x=2.5em,y=-2em,baseline={([yshift=-2ex]current bounding box.center)}]
\Bbox[0]{2}{1}{2}{$\infFT_{p_{k}}$}
\end{tikzpicture}
\qq{=}
\begin{tikzpicture}[x=2.5em,y=-2em,baseline={([yshift=-2ex]current bounding box.center)}]
\Bbox[0]{2}{1}{2}{$\FT^\ell$}
\Bbox[1]{2}{1}{2}{$\alpha$}
\end{tikzpicture}
\]
for some $\ell\geq 0$ and some color-pure braid $\alpha$.  Note that we must have $\ell\rightarrow\infty$ as $k\rightarrow\infty$.  Meanwhile, Proposition \ref{prop:colorpure complex} shows that $\delta_x$ has an intermediate coloring $\gamma_x$ with $\cs{\gamma_x}\leq\cs{\gamma}-1$.  We use this fact to write $\delta_x$ as a concatenation at this intermediate coloring:
\[
\begin{tikzpicture}[x=2em,y=-2em,baseline={([yshift=-2ex]current bounding box.center)}]
\Bbox[0]{2}{1}{2}{$\delta_x$}
\smnode at (-.3,0){$\gamma$};
\smnode at (-.3,1){$\gamma$};
\end{tikzpicture}
\qq{=}
\begin{tikzpicture}[x=2em,y=-2em,baseline={([yshift=-2ex]current bounding box.center)}]
\Bbox[0]{2}{1}{2}{$\delta_x^{top}$}
\Bbox[1]{2}{1}{2}{$\delta_x^{bot}$}
\smnode at (-.3,0){$\gamma$};
\smnode at (-.3,1){$\gamma_x$};
\smnode at (-.3,2){$\gamma$};
\end{tikzpicture}
\]
where the colorings $\gamma$ and $\gamma_x$ are indicated at various points of the diagram.  Thus, our multicone for $\KC{\infFT_{p_k}} \otimes \Cch$ is comprised entirely of complexes of the form
\[
\begin{tikzpicture}[x=2em,y=-2em,baseline={([yshift=-2ex]current bounding box.center)}]
\Bbox[0]{2}{1}{2}{$\FT^\ell$}
\Bbox[1]{2}{1}{2}{$\alpha$}
\Bbox[2]{2}{1}{2}{$\delta_x^{top}$}
\Bbox[3]{2}{1}{2}{$\delta_x^{bot}$}
\smnode at (-.3,0){$\gamma$};
\smnode at (-.3,2){$\gamma$};
\smnode at (-.3,3){$\gamma_x$};
\smnode at (-.3,4){$\gamma$};
\end{tikzpicture}\quad .
\]

Now we use the fact, described in Remark 3.17 of \cite{MAMW}, that the web-braid diagram $\alpha\cdot \delta_x^{top}$ commutes with the full twists via a braid-like isotopy that `pulls rungs through the interior of the torus' as illustrated for a single rung in Figure \ref{fig:pull rung thru twist} (pulling through $\alpha$ is merely restating the well-known fact that the full twist is central in the braid group).  Thus we have
\[
\begin{tikzpicture}[x=2em,y=-2em,baseline={([yshift=-2ex]current bounding box.center)}]
\Bbox[0]{2}{1}{2}{$\FT^\ell$}
\Bbox[1]{2}{1}{2}{$\alpha$}
\Bbox[2]{2}{1}{2}{$\delta_x^{top}$}
\Bbox[3]{2}{1}{2}{$\delta_x^{bot}$}
\smnode at (-.3,0){$\gamma$};
\smnode at (-.3,2){$\gamma$};
\smnode at (-.3,3){$\gamma_x$};
\smnode at (-.3,4){$\gamma$};
\end{tikzpicture}
\qq{\cong}
\begin{tikzpicture}[x=2em,y=-2em,baseline={([yshift=-2ex]current bounding box.center)}]
\Bbox[0]{2}{1}{2}{$\alpha$}
\Bbox[1]{2}{1}{2}{$\delta_x^{top}$}
\Bbox[2]{2}{1}{2}{$\FT^\ell$}
\Bbox[3]{2}{1}{2}{$\delta_x^{bot}$}
\smnode at (-.3,0){$\gamma$};
\smnode at (-.3,1){$\gamma$};
\smnode at (-.3,2){$\gamma_x$};
\smnode at (-.3,3){$\gamma_x$};
\smnode at (-.3,4){$\gamma$};
\end{tikzpicture}
\]

\begin{figure}
\centering
\scalebox{-1}[-1]{\includegraphics[scale=.4]{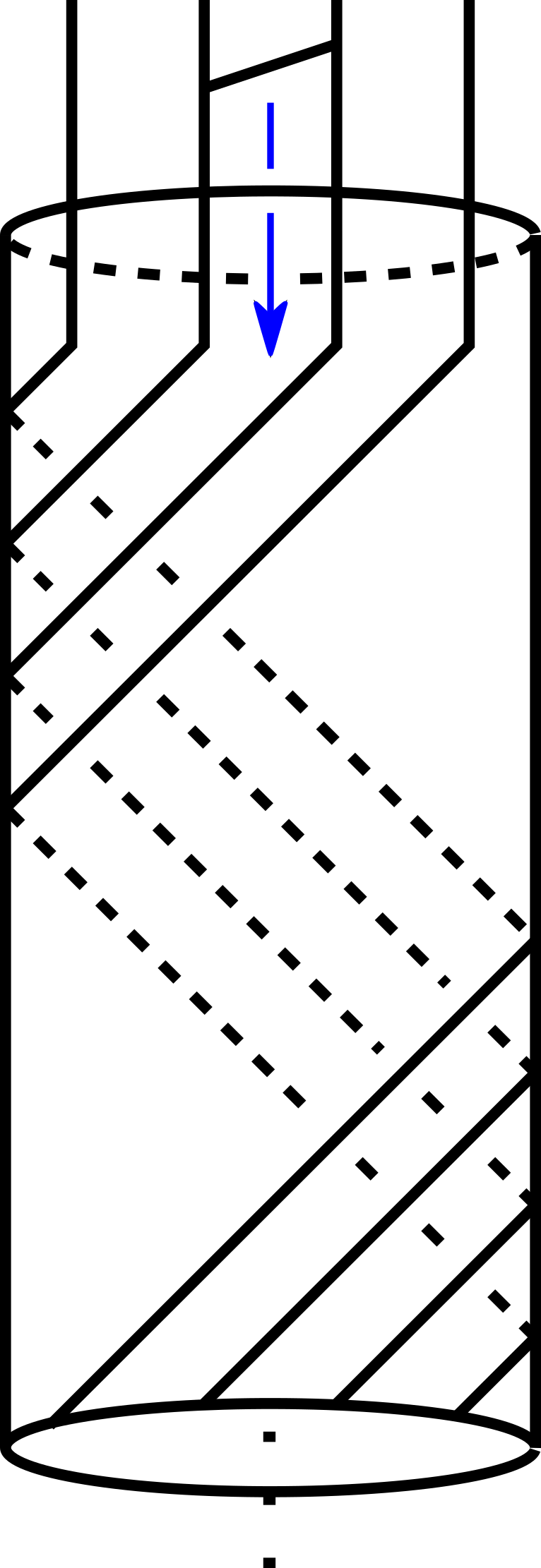}}
\caption{An illustration of pulling a rung through a full twist $\FT$ via passing through the interior of the torus defining $\FT$.  Since the full twist is central in the braid group, we conclude that any web-braid diagram commutes with $\FT$.}
\label{fig:pull rung thru twist}
\end{figure} 

This isotopy produces a homological shift $t_x$ according to Proposition \ref{prop:isotopyShifts} that can be estimated with the help of Proposition \ref{prop:FT colorsize}.  Since all of the crossings in the web-braid remain colored the same except those within the full twist, we have
\begin{align*}
t_x &= \sum_{\tau\in\FT^\ell_{(\gamma)}} \min(\tau) - \sum_{\chi\in\FT^\ell_{(\gamma_x)}} \min(\chi)\\
&= 2\ell(\cs{\gamma}-\cs{\gamma_x})\\
&\geq 2\ell.
\end{align*}

As such, every complex within the multicone expansion for $\KC{\infFT_{p_k}} \otimes \Cch$ can be replaced (up to chain homotopy equivalence, via Proposition \ref{prop:gen multicone equiv}) by one of the form $\h^{t_x}\KC{\alpha\cdot\delta_x^{top}\cdot\FT^\ell\cdot\delta_x^{bot}}$, with minimum non-zero homological degree bounded below by $t_x$ (all crossings in the web-braid diagram are positive).  Thus any diagram $\epsilon$ coming from such a complex within our multicone for $\KC{\infFT_{p_k}} \otimes \Cch$ has overall homological degree estimated with the help of Equation \eqref{eq:gen multicone hom deg}:
\[\h(\epsilon) \geq \h_C(\delta_x) + t_x \geq 2\ell,\]
where we have used Proposition \ref{prop:colorpure complex} to ensure that $\h_C(\delta_x)\geq 0$.  As such, our complex $\KC{\infFT_{p_k}} \otimes \Cch$ is chain homotopy equivalent to one supported in homological degrees greater than or equal to $2\ell$.  Then since $\cone(f_k)\simeq\KC{\infFT_{p_k}} \otimes \Xch$ is a direct summand of this complex, we can conclude that
\[|f_k|_\h \geq 2\ell\]
and as mentioned above, $\ell\rightarrow\infty$ as $k\rightarrow\infty$.  This completes the proof that the inverse system $\{\KC{\infFT_{p_k}},f_k\}$ is Cauchy.

In the case when the ground ring is $\mathbb{C}$, Cautis' inverse system amounts to taking the limit using a subsequence of the maximal purity sequence for the infinite twist, and so Proposition \ref{prop:colorpure inv sys} shows that our limiting complex is chain homotopy equivalent to his.
\end{proof}



\subsection{Color-complete semi-infinite braid words}\label{sec:color complete}

Informally, a color-pure semi-infinite braid word $\B$ is called color-complete if one could arrive at the infinite full twist $\FT^\infty$ by deleting some color-pure braid words from the infinite word $\B$.  In order to make this notion more precise, we need to establish some notation.

\begin{definition}\label{def:sub-braid colorings, deletions}
Let $\B=\sigma_{i_1}\sigma_{i_2}\cdots$ denote a positive semi-infinite braid word that is color-pure with respect to $\gamma$.  Given $a<b$, the notation $\Bminus{\B} a b$ will indicate the positive semi-infinite braid word
\[\Bminus{\B}ab:= \sigma_{i_1}\sigma_{i_2}\cdots \sigma_{i_a} \sigma_{i_{b+1}}\sigma_{i_{b+2}}\cdots\]
derived from $\B$ by deleting the partial sub-braid $\B^a_b$.  Similarly, we will use the notation $\Bminusss{\B}{a_1}{b_1}{a_2}{b_2}$ to denote the word derived from $\B$ by deleting multiple (possibly infinitely many) partial sub-braids.  See Figure \ref{fig:partial sub-braids, colorings, deletions} for a visual illustration.
\end{definition}

\begin{lemma}\label{lem:delete colorpure sub-braid}
If $\B$ is a positive semi-infinite braid word that is color-pure with respect to $\gamma$, and the partial sub-braid $\B^a_b$ is color-pure with respect to $\gamma(a)$, then $\Bminus\B ab$ is also a positive semi-infinite braid word that is color-pure with respect to $\gamma$.
\end{lemma}
\begin{proof} Deleting a color-pure sub-braid within an infinite braid word does not affect the colorings either before or after it.  However, the maximal purity sequence may be shifted beyond the point where $\B^a_b$ was deleted.  See Figure \ref{fig:partial sub-braids, colorings, deletions}.  Details are left to the reader.
\end{proof}

\begin{figure}
\[
\begin{tikzpicture}[x=2em,y=-2em,baseline={([yshift=-2ex]current bounding box.center)}]
\Bbox[0]{2}{1}{2}{$\B_\ell$}
\smnode at (-.4,0){$\gamma$};
\smnode at (-.4,1){$\gamma(\ell)$};
\end{tikzpicture}
\qq{=}
\begin{tikzpicture}[x=2em,y=-2em,baseline={([yshift=-2ex]current bounding box.center)}]
\Bbox[0]{2}{1}{2}{$\B^0_a$}
\Bbox[1]{2}{1}{2}{$\B^a_b$}
\Bbox[2]{2}{1}{2}{$\B^b_c$}
\Bsigma[3]{2}{0}
\Bbox[4]{2}{1}{2}{$\B^e_\ell$}
\smnode at (.5,3.5){$\vdots$};
\smnode at (-.4,0){$\gamma(0)$};
\smnode at (-.4,1){$\gamma(a)$};
\smnode at (-.4,2){$\gamma(b)$};
\smnode at (-.4,3){$\gamma(c)$};
\smnode at (-.4,5){$\gamma(\ell)$};
\end{tikzpicture}
\qquad
\qquad
\qquad
\Bminus{\B}ab \qq{:=}
\begin{tikzpicture}[x=2em,y=-2em,baseline={([yshift=-2ex]current bounding box.center)}]
\Bbox[0]{2}{1}{2}{$\B^0_a$}
\Bbox[1]{2}{1}{2}{$\B^b_c$}
\Bsigma[2]{2}{0}
\smnode at (.5,2.5){$\vdots$};
\smnode at (1.4,0){$\gamma$};
\smnode at (2,1){$\gamma(a)=\gamma(b)$};
\smnode at (1.4,2){$\gamma(c)$};
\end{tikzpicture}
\]
\caption{A partial braid $\B_\ell$ expanded as the product of several partial sub-braids along a sequence $0<a<b<\cdots<e<\ell$.  If it so happens that $\B^a_b$ is color-pure with respect to $\gamma(a)$, so that $\gamma(a)=\gamma(b)$, then we can also define the (still color-pure) semi-infinite braid word $\Bminus{\B}{a}{b}$ by deleting $\B^a_b$ from $\B$.}
\label{fig:partial sub-braids, colorings, deletions}
\end{figure}
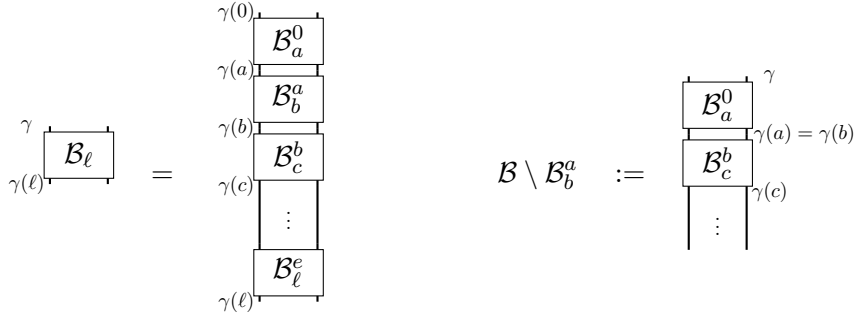

\begin{definition}\label{def:color complete}
A positive color-pure semi-infinite braid word $\B_{(\gamma)}$ is called \emph{color-complete} if there exists a (finite or infinite) sequence of natural numbers (occurring in pairs $a_i,b_i$)
\[a_1<b_1<a_2<b_2<\cdots\]
such that the following conditions hold.
\begin{itemize}
\item Each partial sub-braid $\B^{a_i}_{b_i}$ is color-pure with respect to $\gamma(a_i)$ (that is, $\gamma(a_i)=\gamma(b_i)$).
\item The semi-infinite braid word $\Bminusss{\B}{a_1}{b_1}{a_2}{b_2}$ is (braid move equivalent to) the infinite full twist $\FT^\infty$.
\end{itemize}
This definition is preserved under finitely many allowable braid moves, and so we can consider a positive color-pure semi-infinite braid $\BB_{(\gamma)}$ to be \emph{color-complete} if any positive word representing it is color-complete.
\end{definition}

\subsection{The main result for braid words}\label{sec:proving main thm}

With the language of the preceding sections in place, we can state and prove our main result on the level of braid words.

\begin{theorem}\label{thm:main thm precise}
Fix $n\in\N$, and let $\gamma=(\gamma_1,\dots,\gamma_n)$ be any coloring in $\{0,1,\dots,N\}^n$.  Let $\B_{(\gamma)}$ be any color-complete positive semi-infinite braid word with maximal purity sequence $m_1<m_2<\cdots$.  Then the corresponding inverse system $\{\KC{\B_{m_\ell}},g_\ell\}$ coming from Proposition \ref{prop:colorpure inv sys} is Cauchy, with limiting complex $\KC{\B_{(\gamma)}}$ satisfying
\[\KC{\B_{(\gamma)}} \simeq P_{n,(\gamma)}\]
where $P_{n,(\gamma)}$ is the complex for the infinite full twist defined in Theorem \ref{thm:infFT with max purity seq}.
\end{theorem}
We will prove this theorem using Proposition \ref{prop:compSys}.  Having fixed $\gamma$, the infinite twist $\infFT:=\FT^\infty_{(\gamma)}$ has some maximal purity sequence $p_1<p_2<\cdots$ giving rise to a Cauchy inverse system $\{\KC{\infFT_{p_k}},f_k\}$ such that the limiting complex $\KC{\infFT}$ satisfies $\KC{\infFT}\simeq P_{n,(\gamma)}$ (Theorem \ref{thm:infFT with max purity seq}).  In order to use Proposition \ref{prop:compSys}, we must first construct `horizontal' maps $F_\ell:\KC{\B_{m_\ell}}\rightarrow\KC{\infFT_{p_k}}$ which commute with the system maps, and then we will need to show that $|F_\ell|_\h\rightarrow\infty$ as $\ell\rightarrow\infty$.  We split these tasks into subsections for organizational clarity below.

\subsubsection{Constructing the maps $F_\ell:\KC{\B_{m_\ell}}\rightarrow\KC{\infFT_{p_k}}$}\label{sec:constructing Fell}

The color-completeness of $\B$ implies that, if we fix some $\ell\in\N$, the partial braid word $\B_{m_{\ell}}$ contains some color-pure partial sub-braids $\B^{a_1}_{b_1},\B^{a_2}_{b_2},\dots \B^{a_r}_{b_r}$ such that 
\[\Bminuss{(\B_{m_\ell})}{a_1}{b_1}{a_r}{b_r} \cong \infFT_{p_k},\]
where $k=z(\ell)$ for some non-decreasing function $z:\N\rightarrow\N$ such that $z(\ell)\rightarrow \infty$ as $\ell\rightarrow\infty$ (this will be precisely the indexing function required by Proposition \ref{prop:compSys}).  The index $r$ depends on $\ell$ as well, but this dependence will be irrelevant moving forward.  In an attempt to mimic the proof of Proposition \ref{prop:colorpure complex}, we will use the notation
\[\beta_1:=\B^{a_1}_{b_1}, \beta_2:=\B^{a_2}_{b_2}, \dots \beta_r:=\B^{a_r}_{b_r}\]
which allows the following visual presentation of $\B_{m_\ell}$:
\[
\begin{tikzpicture}[x=2em,y=-2em,baseline={([yshift=-2ex]current bounding box.center)}]
\Bboxst[0]{$\B_{m_\ell}$}
\end{tikzpicture}
\qq{\cong}
\begin{tikzpicture}[x=2em,y=-2em,baseline={([yshift=-2ex]current bounding box.center)}]
\Bboxst[0]{$\infFT^0_{q_1}$}
\Bboxst[1]{$\beta_1$}
\Bboxst[2]{$\infFT^{q_1}_{q_2}$}
\Bboxst[3]{$\beta_2$}
\Bdots[4]{2}{1}
\Bboxst[5]{$\beta_r$}
\Bboxst[6]{$\infFT^{q_r}_{p_k}$}
\end{tikzpicture}
\]
where the indices $q_i$ depend on the exact word $\B$, but will not concern us.  The important point is that $\B_{m_\ell}$ can be presented in a piecewise fashion as indicated.

Now Proposition \ref{prop:colorpure complex} ensures that, for each $i=1,\dots,r$, we have
\[\KC{\beta_i}\simeq \cone(I\rightarrow\Xch_i)\]
where $\Xch_i$ is a direct summand of a complex $\Cch_i$ with specified properties.  We begin by applying this fact to each of the $\beta_i$ in $\B_{m_\ell}$ in an iterative fashion.  If we follow along the reasoning presented in the proof of Proposition \ref{prop:colorpure complex}, we arrive at the following iterated cone presentation for $\KC{\B_{m_\ell}}$ (we have omitted the word `cone' from the notation to avoid clutter):
\[
\begin{tikzpicture}[x=2em,y=-2em,baseline={([yshift=-2ex]current bounding box.center)}]
\Bboxst[0]{$\B_{m_\ell}$}
\end{tikzpicture}
\simeq
\left( 
\left( \cdots 
\left( 
\left( 
\begin{tikzpicture}[x=2.2em,y=-2em,baseline={([yshift=-2ex]current bounding box.center)}]
\Bboxst[0]{$\infFT^0_{p_k}$}
\end{tikzpicture}
\rightarrow 
\begin{tikzpicture}[x=2.2em,y=-2em,baseline={([yshift=-2ex]current bounding box.center)}]
\Bboxst[0]{$\infFT^0_{q_{r}}$}
\Bboxst[1]{$\Xch_r$}
\Bboxst[2]{$\infFT^{q_r}_{p_k}$}
\end{tikzpicture}
\right)
\rightarrow
\begin{tikzpicture}[x=2.2em,y=-2em,baseline={([yshift=-2ex]current bounding box.center)}]
\Bboxst[0]{$\infFT^0_{q_{r-1}}$}
\Bboxst[1]{$\Xch_{r-1}$}
\Bboxst[2]{$\infFT^{q_{r-1}}_{q_r}$}
\Bboxst[3]{$\beta_r$}
\Bboxst[4]{$\infFT^{q_r}_{p_k}$}
\end{tikzpicture}
\right)
\cdots \rightarrow
\begin{tikzpicture}[x=2.2em,y=-2em,baseline={([yshift=-2ex]current bounding box.center)}]
\Bboxst[0]{$\infFT^0_{q_2}$}
\Bboxst[1]{$\Xch_2$}
\Bboxst[2]{$\infFT^{q_2}_{q_3}$}
\Bboxst[3]{$\beta_3$}
\Bdots[4]{2}{1}
\Bboxst[5]{$\beta_r$}
\Bboxst[6]{$\infFT^{q_r}_{p_k}$}
\end{tikzpicture}
\right)
\rightarrow
\begin{tikzpicture}[x=2.2em,y=-2em,baseline={([yshift=-2ex]current bounding box.center)}]
\Bboxst[0]{$\infFT^0_{q_1}$}
\Bboxst[1]{$\Xch_1$}
\Bboxst[2]{$\infFT^{q_1}_{q_2}$}
\Bboxst[3]{$\beta_2$}
\Bdots[4]{2}{1}
\Bboxst[5]{$\beta_r$}
\Bboxst[6]{$\infFT^{q_r}_{p_k}$}
\end{tikzpicture}
\right).
\]

We again collapse this into a single cone
\begin{equation}\label{eq:basic KC(Bml) cone}
\KC{\B_{m_\ell}} \simeq \cone( \KC{\infFT_{p_k}} \rightarrow \td{\Xch} )
\end{equation}
where $\td{\Xch}:=\bigoplus_{i=1}^r \td{\Xch}_i$ is a direct summand of the complex $\td{\Cch}:=\bigoplus_{i=1}^r \td{\Cch}_i$ with the definitions of $\td{\Xch}_i$ and $\td{\Cch}_i$ illustrated below:
\[
\td{\Xch}_i:=
\begin{tikzpicture}[x=2.2em,y=-2em,baseline={([yshift=-2ex]current bounding box.center)}]
\Bboxst[0]{$\infFT^0_{q_i}$}
\Bboxst[1]{$\Xch_i$}
\Bboxst[2]{$\infFT^{q_i}_{q_{i+1}}$}
\Bboxst[3]{$\beta_{i+1}$}
\Bdots[4]{2}{1}
\Bboxst[5]{$\beta_r$}
\Bboxst[6]{$\infFT^{q_r}_{p_k}$}
\end{tikzpicture}
\qq{,}\qq{}
\td{\Cch}_i:=
\begin{tikzpicture}[x=2.2em,y=-2em,baseline={([yshift=-2ex]current bounding box.center)}]
\Bboxst[0]{$\infFT^0_{q_i}$}
\Bboxst[1]{$\Cch_i$}
\Bboxst[2]{$\infFT^{q_i}_{q_{i+1}}$}
\Bboxst[3]{$\beta_{i+1}$}
\Bdots[4]{2}{1}
\Bboxst[5]{$\beta_r$}
\Bboxst[6]{$\infFT^{q_r}_{p_k}$}
\end{tikzpicture}.
\]
This allows us to define $F_\ell:\KC{\B_{m_\ell}} \rightarrow \KC{\infFT_{p_k}}$ via the quotient map implied by Equation \eqref{eq:basic KC(Bml) cone}.  Since all of the system maps are built as similar quotient maps corresponding to deleting various partial sub-braids, it is clear that $F_\ell$ commutes with the system maps $f_k,g_\ell$ as required.

\subsubsection{Estimating $|F_\ell|_\h$ with a lower bound $b(\ell)$}
\label{sec:bounding Fell}

In order to estimate $|F_\ell|_\h$, we need to understand $\cone(F_\ell)\simeq\td{\Xch}$.  Mimicking the proof of Theorem \ref{thm:infFT with max purity seq}, we consider the complex $\td{\Cch}$ instead.  We examine a single summand $\td{\Cch}_i$ of $\td{\Cch}$ by expanding it as a large iterated multicone along each $\KCsimp{\beta_j}$ for $i<j\leq r$ (recall the notation $\KCsimp{\cdot}$ for the simplified cone complex of Proposition \ref{prop:colorpure complex}).  We illustrate the first two steps in Figure \ref{fig:iterating the mcone}.

\begin{figure}
\begin{align*}
\td{\Cch}_i &\simeq
\underset{\delta_{x},\delta_{y}\in\KCsimp{\beta_{i+1}}}{\Mcone}\left( 
\begin{tikzpicture}[x=2.2em,y=-2em,baseline={([yshift=-2ex]current bounding box.center)}]
\Bboxst[0]{$\infFT^0_{q_i}$}
\Bboxst[1]{$\Cch_i$}
\Bboxst[2]{$\infFT^{q_i}_{q_{i+1}}$}
\Bboxst[3]{$\delta_{x}$}
\Bdots[4]{2}{1}
\Bboxst[5]{$\infFT^{q_r}_{p_k}$}
\end{tikzpicture}
\qq{\rightarrow}
\begin{tikzpicture}[x=2.2em,y=-2em,baseline={([yshift=-2ex]current bounding box.center)}]
\Bboxst[0]{$\infFT^0_{q_i}$}
\Bboxst[1]{$\Cch_i$}
\Bboxst[2]{$\infFT^{q_i}_{q_{i+1}}$}
\Bboxst[3]{$\delta_{y}$}
\Bdots[4]{2}{1}
\Bboxst[5]{$\infFT^{q_r}_{p_k}$}
\end{tikzpicture}
\right)
\\
&\simeq
\underset{\delta_{x},\delta_{y}\in\KCsimp{\beta_{i+1}}}{\Mcone}\left( 
\underset{\delta'_{x},\delta'_{y}\in\KCsimp{\beta_{i+2}}}{\Mcone}\left( 
\begin{tikzpicture}[x=2.2em,y=-2em,baseline={([yshift=-2ex]current bounding box.center)}]
\Bboxst[0]{$\infFT^0_{q_i}$}
\Bboxst[1]{$\Cch_i$}
\Bboxst[2]{$\infFT^{q_i}_{q_{i+1}}$}
\Bboxst[3]{$\delta_{x}$}
\Bboxst[4]{$\infFT^{q_{i+1}}_{q_{i+2}}$}
\Bboxst[5]{$\delta'_{x}$}
\Bdots[6]{2}{1}
\Bboxst[7]{$\infFT^{q_r}_{p_k}$}
\end{tikzpicture}
{\rightarrow}
\begin{tikzpicture}[x=2.2em,y=-2em,baseline={([yshift=-2ex]current bounding box.center)}]
\Bboxst[0]{$\infFT^0_{q_i}$}
\Bboxst[1]{$\Cch_i$}
\Bboxst[2]{$\infFT^{q_i}_{q_{i+1}}$}
\Bboxst[3]{$\delta_{x}$}
\Bboxst[4]{$\infFT^{q_{i+1}}_{q_{i+2}}$}
\Bboxst[5]{$\delta'_{y}$}
\Bdots[6]{2}{1}
\Bboxst[7]{$\infFT^{q_r}_{p_k}$}
\end{tikzpicture}
\right)
{\rightarrow}
\underset{\delta'_{x},\delta'_{y}\in\KCsimp{\beta_{i+2}}}{\Mcone}\left( 
\begin{tikzpicture}[x=2.2em,y=-2em,baseline={([yshift=-2ex]current bounding box.center)}]
\Bboxst[0]{$\infFT^0_{q_i}$}
\Bboxst[1]{$\Cch_i$}
\Bboxst[2]{$\infFT^{q_i}_{q_{i+1}}$}
\Bboxst[3]{$\delta_{y}$}
\Bboxst[4]{$\infFT^{q_{i+1}}_{q_{i+2}}$}
\Bboxst[5]{$\delta'_{x}$}
\Bdots[6]{2}{1}
\Bboxst[7]{$\infFT^{q_r}_{p_k}$}
\end{tikzpicture}
{\rightarrow}
\begin{tikzpicture}[x=2.2em,y=-2em,baseline={([yshift=-2ex]current bounding box.center)}]
\Bboxst[0]{$\infFT^0_{q_i}$}
\Bboxst[1]{$\Cch_i$}
\Bboxst[2]{$\infFT^{q_i}_{q_{i+1}}$}
\Bboxst[3]{$\delta_{y}$}
\Bboxst[4]{$\infFT^{q_{i+1}}_{q_{i+2}}$}
\Bboxst[5]{$\delta'_{y}$}
\Bdots[6]{2}{1}
\Bboxst[7]{$\infFT^{q_r}_{p_k}$}
\end{tikzpicture}
\right)
\right)\\
&\simeq \cdots
\end{align*}
\caption{The first two steps expanding $\Cch_i$ as an iterated multicone.  The vertical ellipses here are meant to indicate that the remainder of the diagram below the lowest $\delta$ matches the diagram defining $\td{\Cch}_i$ below the corresponding point.  The next step would be to expand all four of the terms visualized in the bottom row along $\KCsimp{\beta_{i+3}}$, and we would continue in this fashion until there are no $\beta_j$ left.}
\label{fig:iterating the mcone}
\end{figure}

In the end we may collapse the resulting iterated multicone into one large multicone incorporating all of the possible complexes.  In order to write this cleanly, we introduce some notation.

\begin{definition}
We will write $\KCsimp{\beta}$ as a shorthand to indicate the formal tensor product $\bigotimes_{j=i+1}^r \KCsimp{\beta_j}$, with objects $\delta:=\delta_{i+1}\otimes\cdots\otimes\delta_r$.  To the formal object $\delta$ we associate a diagram also denoted $\delta$ formed by placing each $\delta_j$ in place of $\beta_j$ in the diagram for $\td{\Cch}_i$ as illustrated below
\[
\delta=\delta_{i+1}\otimes\cdots\otimes\delta_r \qq{\Longrightarrow}
\delta:=
\begin{tikzpicture}[x=2.2em,y=-2em,baseline={([yshift=-2ex]current bounding box.center)}]
\Bboxst[0]{$\infFT^0_{q_i}$}
\Bboxst[1]{$\Cch_i$}
\Bboxst[2]{$\infFT^{q_i}_{q_{i+1}}$}
\Bboxst[3]{$\delta_{i+1}$}
\Bboxst[4]{$\infFT^{q_{i+1}}_{q_{i+2}}$}
\Bdots[5]{2}{1}
\Bboxst[6]{$\delta_r$}
\Bboxst[7]{$\infFT^{q_r}_{p_k}$}
\end{tikzpicture}
\qq{,} \text{for diagrams $\delta_j\in\KCsimp{\beta_j}$ for $i<j\leq r$}.
\]
Maps between two formal objects $\delta_x,\delta_y\in\KCsimp{\beta}$ then induce maps between the diagram complexes $\KC{\delta_x},\KC{\delta_y}$ formed by stitching together cobordisms in the usual way.
\end{definition}

With this notation we can collapse our iterated multicone into a single multicone 
\begin{equation}\label{eq:Ci first mcone}
\td{\Cch}_i \simeq
\underset{\delta_x,\delta_y\in\KCsimp{\beta}}{\Mcone}\left(
\KC{\delta_x} \longrightarrow \KC{\delta_y} \right),
\end{equation}
for which we have the formula
\begin{equation}\label{eq:hom deg of delta in beta}
\h_{\KCsimp{\beta}}(\delta) = \sum_{j=i+1}^r \h_{\KCsimp{\beta_j}}(\delta_j).
\end{equation}

From here, we analyze the single complexes within this multicone \eqref{eq:Ci first mcone} in a manner very similar to the proof of Theorem 3.3 in \cite{MAMW}.  We begin by fixing such a complex $\KC{\delta}$, and we define two quantities based on $\delta$.  First, we have
\begin{equation}
c_1(\delta) := \big|\{j\in\{i+1,\dots,r\} \,|\, \delta_j\in\Xch_j\} \big|.
\end{equation}
Since any $\delta_j\in\Xch_j$ must have $\h_{\KCsimp{\beta_j}}(\delta_j)>0$, Equation \eqref{eq:hom deg of delta in beta} quickly produces the bound
\begin{equation}\label{eq:c1 bound}
\h_{\KCsimp{\beta}}(\delta) \geq c_1(\delta).
\end{equation}
Second, we define
\begin{equation}
c_2(\delta) := \text{largest number of \emph{uninterrupted} full twists in the diagram $\delta$},
\end{equation}
and we make the following claim.

\begin{lemma}\label{lem:KC simeq Y}
Any complex $\KC{\delta}$ of the form indicated above is chain homotopy equivalent to a complex $\Ych(\delta)$ having no terms in homological degree less than $2c_2(\delta)$.
\end{lemma}
\begin{proof}
Indeed if we expand $\KC{\delta}$ as its own multicone along the complex $\Cch_i$, we see (denoting the terms in $\Cch_i$ by $\delta_i$ parallel to the notation for $\delta_j\in\KCsimp{\beta_j}$)
\[
\KC{\delta} \simeq
\underset{\delta_{i,x},\delta_{i,y}\in\Cch_i}{\Mcone}\left( 
\begin{tikzpicture}[x=2.2em,y=-2em,baseline={([yshift=-2ex]current bounding box.center)}]
\Bboxst[0]{$\infFT^0_{q_i}$}
\Bboxst[1]{$\delta_{i,x}$}
\Bboxst[2]{$\infFT^{q_i}_{q_{i+1}}$}
\Bboxst[3]{$\delta_{i+1}$}
\Bdots[4]{2}{1}
\Bboxst[5]{$\infFT^{q_r}_{p_k}$}
\end{tikzpicture}
\qq{\rightarrow}
\begin{tikzpicture}[x=2.2em,y=-2em,baseline={([yshift=-2ex]current bounding box.center)}]
\Bboxst[0]{$\infFT^0_{q_i}$}
\Bboxst[1]{$\delta_{i,y}$}
\Bboxst[2]{$\infFT^{q_i}_{q_{i+1}}$}
\Bboxst[3]{$\delta_{i+1}$}
\Bdots[4]{2}{1}
\Bboxst[5]{$\infFT^{q_r}_{p_k}$}
\end{tikzpicture}
\right).
\]
Recall from Proposition \ref{prop:colorpure complex} that any $\delta_{i,x}\in\Cch_i$ is a ladder diagram containing an intermediate coloring $\gamma_x$ with $\cs{\gamma_x}<\cs{\gamma(q_i)}$, so that we may split $\delta_{i,x}$ just as in the proof of Theorem \ref{thm:infFT with max purity seq}.  We also have $c_2(\delta)$ full twists available somewhere in the diagram (independently of the choice of $\delta_{i,x}$) which commute with any and all web-braids (recall Figure \ref{fig:pull rung thru twist}), and so we can perform the following moves on any single complex within the multicone:
\begin{equation}\label{eq:each diagram in delta multicone}
\KC{
\begin{tikzpicture}[x=2.2em,y=-2em,baseline={([yshift=-2ex]current bounding box.center)}]
\Bboxst[0]{$\infFT^0_{q_i}$}
\Bboxst[1]{$\delta_{i,x}^{top}$}
\Bboxst[2]{$\delta_{i,x}^{bot}$}
\Bboxst[3]{$\alpha$}
\Bboxst[4]{$\FT^{c_2(\delta)}$}
\Bboxst[5]{$\alpha'$}
\smnode at (-.5,0){$\gamma$};
\smnode at (-.5,1){$\gamma(q_i)$};
\smnode at (-.5,2){$\gamma_x$};
\smnode at (-.5,3){$\gamma(q_i)$};
\smnode at (-.5,4){$\hat{\gamma}$};
\smnode at (-.5,5){$\hat{\gamma}$};
\smnode at (-.5,6){$\gamma$};
\end{tikzpicture}
}
\qq{\simeq}
\h^{t_x}\KC{
\begin{tikzpicture}[x=2.2em,y=-2em,baseline={([yshift=-2ex]current bounding box.center)}]
\Bboxst[0]{$\infFT^0_{q_i}$}
\Bboxst[1]{$\delta_{i,x}^{top}$}
\Bboxst[2]{$\FT^{c_2(\delta)}$}
\Bboxst[3]{$\delta_{i,x}^{bot}$}
\Bboxst[4]{$\alpha$}
\Bboxst[5]{$\alpha'$}
\smnode at (-.5,0){$\gamma$};
\smnode at (-.5,1){$\gamma(q_i)$};
\smnode at (-.5,2){$\gamma_x$};
\smnode at (-.5,3){$\gamma_x$};
\smnode at (-.5,4){$\gamma(q_i)$};
\smnode at (-.5,5){$\hat{\gamma}$};
\smnode at (-.5,6){$\gamma$};
\end{tikzpicture}
}.
\end{equation}
In these diagrams, $\alpha$ and $\alpha'$ are web-braid diagrams consisting of various partial sub-braids $\infFT^a_b$ together with various $\delta_j$ depending on the precise format of our original diagram $\delta$.  The various intermediate colorings are also indicated.  Since all of the $\delta_j$ came from color-pure diagrams $\beta_j$, and sub-braids can permute colors but not change them, we have some intermediate coloring $\hat{\gamma}$ between $\alpha$ and $\alpha'$ that is a permutation of $\gamma$.  In particular, $\cs{\hat{\gamma}}=\cs{\gamma}>\cs{\gamma_x}$.  Now when we commute the full twists past the other parts of the diagram, we leave all colors on all crossings unchanged except for the crossings of the full twists, and so we have an overall homological shift of $t_x\geq 2c_2(\delta)$ again just as in the proof of Theorem \ref{thm:infFT with max purity seq}.

From here, Proposition \ref{prop:gen multicone equiv} ensures that our multicone for $\KC{\delta}$ is equivalent to another multicone $\Ych(\delta)$ made up entirely of complexes of the form on the right-hand side of Equation \eqref{eq:each diagram in delta multicone}.  Using the formula \eqref{eq:gen multicone hom deg} for homological grading of terms in a multicone, together with the fact that $\h_{\Cch_i}(\delta_{i})\geq 0$ for any $\delta_i\in\Cch_i$ (Proposition \ref{prop:colorpure complex}), we can conclude that $\Ych(\delta)$ indeed has no terms in homological degree below $2c_2(\delta)$.
\end{proof}

We now apply Proposition \ref{prop:gen multicone equiv} to the multicone \eqref{eq:Ci first mcone} to replace the complexes $\KC{\delta}$ with the corresponding complexes $\Ych(\delta)$ from Lemma \ref{lem:KC simeq Y}:
\[\td{\Cch}_i \simeq 
\underset{\delta_x,\delta_y \in \KCsimp{\beta}}{\Mcone}\left( \Ych(\delta_x) \longrightarrow \Ych(\delta_y)\right).\]

Any non-zero term in our multicone for $\td{\Cch}_i$ must be a diagram of the form $\epsilon_\delta$ coming from some $\Ych(\delta)$ in the multicone.  The homological grading of $\epsilon_\delta$ is computed via Equation \eqref{eq:gen multicone hom deg}:
\[\h_{\td{\Cch}_i}(\epsilon_\delta) = \h_{\KCsimp{\beta}}(\delta) +\h_{\Ych(\delta)}(\epsilon_\delta).\]
We then invoke the bound of Equation \eqref{eq:c1 bound} together with our homological condition on the complex $\Ych(\delta)$ (Lemma \ref{lem:KC simeq Y}) to conclude that
\[\h_{\td{\Cch_i}}(\epsilon_\delta) \geq c_1(\delta)+2c_2(\delta)\]
and in turn that the complex $\td{\Cch}=\bigoplus_{i=1}^r \Cch_i$ is supported in homological gradings greater than or equal to the bound $b(\ell)$ defined as
\begin{equation}\label{eq:coneF global bound}
b(\ell):=\min_{i\in\{1,\dots,r\}} \left( \min_{\delta\in\KCsimp{\beta} \text{ in $\Cch_i$}} \left( c_1(\delta)+2c_2(\delta) \right) \right).
\end{equation}
Since $\Xch\simeq\cone(F_\ell)$ is a direct summand of $\Cch$, we can conclude the same about $\cone(F_\ell)$ giving us $|F_\ell|_\h\geq b(\ell)$.




\subsubsection{Finishing the argument}
\begin{proof}[Proof of Theorem \ref{thm:main thm precise}]

We have two inverse systems $\{\KC{\infFT_{p_k}},f_k\}$ and $\{\KC{\B_{m_\ell}},g_\ell\}$ together with maps $F_\ell$ between them creating the commuting diagram of Figure \ref{fig:compare inv sys} (Section \ref{sec:constructing Fell}).  We have also produced a lower bound $b(\ell)$ \eqref{eq:coneF global bound} on the homological orders $|F_\ell|_\h$ (Section \ref{sec:bounding Fell}).

Just as in the proof of Theorem 3.3 in \cite{MAMW}, the reader can quickly verify that the color-completeness of $\B$ ensures that this bound must grow infinite as $\ell\rightarrow\infty$.  Roughly speaking, as $\ell$ grows, so too does $k$ (this is color-completeness), and thus we have an ever-growing number of full twists `available' in any diagram $\delta$ involved in any $\td{\Cch_i}$.  If the full twists are largely uninterrupted, $c_2$ must be large by definition; otherwise we have many non-identity $\delta_j$ diagrams `in the way', and so $c_1$ must be large.  In either case, their sum which defines $b(\ell)$ is growing without bound as $\ell$ grows.

A similar (and simpler) argument shows that the system $\{\KC{\B_{m_\ell}},g_\ell\}$ is Cauchy, and so by Proposition \ref{prop:compSys} we are done.
\end{proof}

\begin{corollary}\label{cor:main result pos braids}
If $\BB_{(\gamma)}$ is a positive color-complete semi-infinite braid, then there is a well-defined limiting system $\KC{\BB_{(\gamma)}}\simeq P_{n,(\gamma)}$ up to chain homotopy equivalence built from the inverse system arising from any positive word representing $\BB_{(\gamma)}$.
\end{corollary}
\begin{proof}
Combine Theorem \ref{thm:main thm precise} and Proposition \ref{prop:colorpure cx well defined}.
\end{proof}

\section{Further corollaries and general results}\label{sec:corollaries}
In this section we record some corollaries of our work that explore a variety of situations.  We begin with a quick corollary that shows how our limiting complexes behave similarly to categorified highest weight projectors.

\begin{corollary}\label{cor:P absorbs colorpure}
If $\beta$ is a positive braid color-pure with respect to some coloring $\gamma$, then
\[\KC{\beta}\otimes P_{n,(\gamma)} \simeq P_{n,(\gamma)} \otimes \KC{\beta} \simeq P_{n,(\gamma)}.\]
\end{corollary}
\begin{proof}
Since $\beta$ is positive and color-pure with respect to $\gamma$, we can view the concatenation $\beta\cdot\FT^\infty$ as a single positive semi-infinite color-complete braid word with limiting complex $P_{n,(\gamma)}$ by Theorem \ref{thm:main thm precise}.  The statement of the corollary then follows from Lemma \ref{lem:Lim of product}.
\end{proof}

\begin{corollary}\label{cor:indep of starting point}
If $\B_{(\gamma)}$ is a positive semi-infinite color-complete braid word with maximal purity sequence $m_1<m_2<\cdots$, then $\KC{\B}\simeq\KC{\B_{m_i}}\otimes\KC{\B^{m_i}_\infty}\simeq\KC{\B^{m_i}_\infty}$ for any $i$.  That is to say, the complex for $\B_{(\gamma)}$ is independent of the `starting point' for the braid word, provided that this starting point is colored by $\gamma$.
\end{corollary}

\subsection{Negative crossings}\label{sec:neg crossings}
In order to deal with semi-infinite braid words involving negative crossings, we require the following proposition that can be seen as a weak generalization of Proposition \ref{prop:colorpure complex}.

\begin{proposition}\label{prop:colorpure posneg complex}
Given a color-pure braid $\beta_{(\gamma)}$ on $n$ strands, let
\begin{equation}\label{eq:neg shift}
t^- := \sum_{\tau^-\in T_1} \min(\tau^-)
\end{equation}
denote the sum of the minimum colors at each \emph{negative} crossing $\tau^-$ in $\beta$.  Then the complex $\Ach:=\KC{\beta_{(\gamma)}}$ is chain homotopy equivalent to a multicone satisfying the following properties.
\begin{itemize}
\item There is a single term corresponding to the identity diagram $I$, and $\h_\Ach(I)=t^-$.
\item Every other term $\delta_x\in\Ach$ contains some intermediate coloring $\gamma_x$ with $\cs{\gamma_x}<\cs{\gamma}$.
\end{itemize}
\end{proposition}
\begin{proof}
This proposition is proved in the same general fashion as Proposition \ref{prop:colorpure complex}, but is much simpler because we are not attempting to isolate the identity diagram.  As such we will be brief.

We begin by expanding $\Ach$ as an iterated multicone along each of the uni-colored crossings as in the logic of Figure \ref{fig:braid as mcone ex}.  The identity resolutions of such crossings sit in right-most homological degree for negative crossings, and so we have the uni-colored negative crossings contributing their part to $t^-$ (and all other resolutions contain intermediate colorings as required).

We then have a corresponding version of Lemma \ref{lem:clasps in colorpure} stating that a color-pure braid with no uni-colored crossings is braid isotopic to one having a positive clasp, a negative clasp, or a Reidemeister II move available.  A Reidemeister II move can be applied, creating a shift by the minimum of the two colors as necessary.  Clasps can be expanded as a tensor product of two copies of Equation \eqref{eq:posxDef} or \eqref{eq:negxDef}; in either case, we have a trapezoid (in either far left or far right homological degree) that can be replaced by a sum of terms including the identity diagram just as in the proof of Lemma \ref{lem:2-color clasp complex}, except this time we do not bother with Gaussian eliminations.  Instead, we simply note that all of the non-identity diagrams, whether in the same homological degree or not, have intermediate colorings as required.  The details of this expansion and the resulting homological placement of the identity diagram are left to the reader.
\end{proof}

Now given a color-pure semi-infinite braid word $\B_{(\gamma)}$ with only finitely many negative crossings, there is some minimal $r$ such that $\B_r$ is color-pure and $\B^r_\infty$ is both color-pure and positive.  Thus $\B^r_\infty$ has a maximal purity sequence $m_0<m_1<\cdots$, which we can use to define an inverse system $\{\KC{\B_{m_k+r}},f_k\}$ for $\B$.

\begin{corollary}\label{cor:fin many negs}
If $\B_{(\gamma)}$ is a color-complete semi-infinite braid with only finitely many crossings, then the inverse system assigned to $\B$ has an inverse limit $\KC{\B_{(\gamma)}}$ satisfying
\[\KC{\B_{(\gamma)}} \simeq \h^{t^-}\q^s P_{n,(\gamma)}\]
where $t^-$ is defined as in Equation \eqref{eq:neg shift} and $s$ is some other shift depending on the negative crossings present in $\B$.
\end{corollary}
\begin{proof}
Let $r$ be as above; $\B_{(\gamma)}$ being color-complete ensures that $(\B^r_\infty)_{(\gamma)}$ is color-complete and positive.  Combining Lemma \ref{lem:Lim of product} and Theorem \ref{thm:main thm precise} we have
\[\KC{\B} \simeq \KC{\B_r}\otimes\KC{\B^r_\infty} \simeq \KC{\B_r}\otimes P_{n,(\gamma)} \simeq \lim \KC{\B_r\cdot\FT^k}\]
where we've used the subsequence of complete full twists rather than the maximal purity sequence for simplicity.

Now we use Proposition \ref{prop:colorpure posneg complex} to expand each $\KC{\B_r\cdot\FT^k}$ as a multicone along $\KC{\B_r}\simeq\Ach$.  Every non-identity term in $\Ach$ has intermediate colorings of lesser color-size than $\gamma$, which means that each such term gets shifted in homological degree by some amount that grows with $k$ just as in the proof of Theorem \ref{thm:infFT with max purity seq}.  Thus each $\KC{\B_r\cdot\FT^k}$ is chain homotopy equivalent to a complex with $\FT^k$ in homological degree $t^-$, and other terms in higher homological degrees (once $k$ is large enough).  This enables one to build quotient maps $F_\ell$ from our inverse system to the system for $\FT^\infty$, whose cones live in larger and larger homological orders, allowing Proposition \ref{prop:compSys} to finish the argument as usual.  The details are left to the reader.
\end{proof}

\begin{corollary}\label{cor:main thm posneg}
Given a positive semi-infinite color-complete braid $\BB_{(\gamma)}$ on $n$ strands, the corresponding limiting complex $\KC{\BB_{(\gamma)}}$ is well-defined up to a degree shift that depends on the negative crossings included in the choice of word representing $\BB$.  Equivalently, a Reidemeister II move on a semi-infinite braid induces the same degree shift on the limiting complex that it would induce on a finite braid (this shift depends on the colorings of the strands).
\end{corollary}
\begin{proof}
This follows from Corollary \ref{cor:fin many negs} since we allow only finitely many Reidemeister moves.
\end{proof}

\subsection{Horizontal splittings}
Corollaries \ref{cor:indep of starting point} and \ref{cor:fin many negs} use finite `vertical' composition of color-pure braids.  We also have a result utilizing `horizontal' composition.  Rather than set up special notation just for this case, we present a simplistic visual version of the relevant result.  The reader can consult Corollary 3.8 in \cite{MAMW} for a more detailed statement in the uni-colored case.

\begin{corollary}\label{cor:horizontal splitting}
Suppose $\B_{(\gamma)}$ is a positive color-pure semi-infinite braid word that can be decomposed as
\[
\B_{(\gamma)} = 
\begin{tikzpicture}[x=2em,y=-4em,baseline={([yshift=-2ex]current bounding box.center)}]
\smnode at (-.3,0){$\gamma$};
\Bbox[0]{7}{1}{7}{$\B^0_r$}
\smnode at (-.3,1){$\gamma$};
\smnode at (.5,1){$\widehat{\gamma_1}$};
\BboxOnly[1]{1}{2}{$\widehat{\B_1}$}
\smnode at (2.5,1){$\widehat{\gamma_2}$};
\BboxOnly[1]{3}{4}{$\widehat{\B_2}$}
\node at (4,1.5){$\cdots$};
\smnode at (5.5,1){$\widehat{\gamma_k}$};
\BboxOnly[1]{6}{7}{$\widehat{\B_k}$}
\end{tikzpicture}
\]
for some finite $r\geq 0$ such that
\begin{itemize}
\item $\B^0_r$ is color-pure with respect to $\gamma$, and
\item each semi-infinite $\widehat{\B_i}$ (on $n_i<n$ strands) is color-complete with respect to the coloring $\widehat{\gamma_i}$ obtained from partitioning $\gamma$ as indicated.
\end{itemize}
Then there is a well-defined limiting complex $\KC{\B_{(\gamma)}}$ satisfying
\[
\KC{\B_{(\gamma)}} \simeq
\begin{tikzpicture}[x=2.75em,y=-2em,baseline={([yshift=-2ex]current bounding box.center)}]
\Bbox[0]{7}{1}{7}{$\KC{\B^0_r}$}
\BboxOnly[1]{1}{2}{$P_{n_1,(\widehat{\gamma_1})}$}
\BboxOnly[1]{3}{4}{$P_{n_2,(\widehat{\gamma_2})}$}
\node at (4,1.5){$\cdots$};
\BboxOnly[1]{6}{7}{$P_{n_k,(\widehat{\gamma_k})}$}
\end{tikzpicture}.
\]
\end{corollary}
\begin{proof}
As in Corollary \ref{cor:fin many negs} we use our value $r$ to define $\KC{\B_{(\gamma)}}:=\KC{\B_r}\otimes\KC{\B^r_\infty}$.  We then appeal to Lemma \ref{lem:Lim of product} and the `horizontal' concatenation properties of $\foamC$ to complete the proof.  See the proof of Corollary 3.8 in \cite{MAMW} for slightly more detail.
\end{proof}

\subsection{Bi-infinite braids}
For semi-infinite braids it was easy to consider a given coloring $\gamma$ as applying to the strands at the `start' of the braid $\B$.  In order to write down a well-defined limiting complex $\KC{\B}$ then, it was required that the coloring at the `end' of the braid $\B$ was also fixed.  Demanding that the `start' and `end' match naturally leads to the definition of color-purity for $\B$, as motivated by the purity (and hence color-purity) of the powers of the full twist.

If we wish to generalize to bi-infinite braids, we will need to alter our approach slightly.

\begin{definition}\label{def:bi-inf braid word}
A \emph{bi-infinite braid word} is a map $\B:\Z\rightarrow\BnGen$; the word is called \emph{positive} if $\im(\B)\subset\BnGen^+$.  Partial words $\B^a_b$ and truncated words $\B^{-\infty}_a,\B^b_\infty$ are defined in the obvious way.  We say $\B$ is \emph{colored from $\gamma$ to $\gamma'$}, and denote it $\Btwocol[\gamma]$, if there exist two integers $\ell_0\leq m_0$, called \emph{starting points}, such that the following properties hold.
\begin{itemize}
\item The partial word $\B^{\ell_0}_{m_0}$ can be colored as $(\B^{\ell_0}_{m_0})^{(\gamma)}_{(\gamma')}$, and neither $\gamma$ nor $\gamma'$ exist as intermediate colorings in $\B^{\ell_0}_{m_0}$.
\item The truncated semi-infinite words $\B^{-\infty}_{\ell_0}$ and $\B^{m_0}_\infty$ are color-pure with respect to $\gamma$ and $\gamma'$, respectively.  (Here color-purity of $\B^{-\infty}_{\ell_0}$ is defined in the obvious way.)
\end{itemize}
In this case we have a function $\gamma(\cdot):\Z\rightarrow\{1,\dots,N\}^n$ of induced colorings between crossings as before, with two maximal purity sequences $\ell_0>\ell_1>\cdots$ and $m_0<m_1<\cdots$ satisfying $\gamma(\ell_i)=\gamma$ and $\gamma(m_i)=\gamma'$ for all $i$.
\end{definition}

Thus we visualize a colored bi-infinite braid word as having some central `starting word', and `growing outwards' from there:
\[
\begin{tikzpicture}[x=2em,y=-2em,baseline={([yshift=-2ex]current bounding box.center)}]
\Bboxst[0]{$\B$}
\smnode at (-.3,0){$\gamma$};
\smnode at (-.3,1){$\gamma'$};
\end{tikzpicture}
\qq{=}
\begin{tikzpicture}[x=2em,y=-2em,baseline={([yshift=-2ex]current bounding box.center)}]
\smnode at (-.3,0){$\gamma$};
\Bboxst[0]{$\B^{-\infty}_{\ell_0}$}
\smnode at (-.3,1){$\gamma$};
\Bboxst[1]{$\B^{\ell_0}_{m_0}$}
\smnode at (-.3,2){$\gamma'$};
\Bboxst[2]{$\B^{m_0}_\infty$}
\smnode at (-.3,3){$\gamma'$};
\end{tikzpicture}.
\]

Now bi-infinite braids should be considered unchanged by finitely many Reidemeister moves as before, but also by shifts in the function $\B$ since there is no well-defined `starting point' for the braid.

\begin{definition}\label{def:bi-inf braid}
Given a bi-infinite braid word $\B$ and a finite $s\in\Z$, the \emph{shifted} braid word $\B[s]:\Z\rightarrow\BnGen$ is the map $\B[s](i):=\B(i-s)$.  Then a \emph{bi-infinite braid} $\BB$ is an equivalence class of bi-infinite braid words up to finitely many braid moves and shifts, and as before we consider $\BB$ \emph{positive} if some word representing it is positive, and we can color $\BB$ if we can color one (and thus all) of its representatives.
\end{definition}

It should be clear that a colored positive bi-infinite braid $\B$ gives rise to inverse systems for $\B^{-\infty}_{\ell_0}$ and $\B^{m_0}_\infty$, which we can concatenate with $\B^{\ell_0}_{m_0}$ to define an inverse system for $\B$.  Such a system would appear to depend on the choice of starting points $\ell_0<m_0$ in general.  Still, this viewpoint makes it clear what color-completeness should mean.

\begin{definition}\label{def:bi-inf color-complete}
A positive colored bi-infinite braid word $\Btwocol[\gamma]$ is called \emph{color-complete} if, for some (and thus any) choice of starting points $\ell_0<m_0$ satisfying the coloring definition, both semi-infinite braid words $\B^{-\infty}_{\ell_0}$ and $\B^{m_0}_\infty$ are color complete with respect to $\gamma$ and $\gamma'$, respectively.  As usual, the positive braid $\BBtwocol$ is color-complete if some (and thus any) representative word for $\BB$ is.
\end{definition}

The following corollary provides a precise version of Theorem \ref{thm:imprecise main thm bi-inf}.

\begin{corollary}\label{cor:main thm bi-inf}
To a positive color-complete bi-infinite braid word $\Btwocol[\gamma]$ we may assign an inverse system with limiting complex 
\[\KC{\Btwocol[\gamma]}\simeq P_{n,(\gamma)} \otimes \KC{\B^{\ell_0}_{m_0}} \otimes P_{n,(\gamma')}\]
that is independent of the choice of starting points $\ell_0<m_0$ up to chain homotopy equivalence.  Thus we may assign a corresponding limiting complex to any positive color-complete bi-infinite braid that is well-defined up to degree shifts allowing for Reidemeister II moves creating (or deleting) finitely many negative crossings in the representative word.
\end{corollary}
\begin{proof}
For a fixed choice of starting points $\ell_0<m_0$ it is clear from Theorem \ref{thm:main thm precise} and Lemma \ref{lem:Lim of product} that our assumptions ensure that the inverse system has limiting complex as described.  If we choose different starting points $\ell'_0<m'_0$, we partition $\B$ into separate pieces using all four starting points.  Depending on the relative positions of the starting points, we will then have a color-pure finite braid available which can be absorbed into one of the limiting complexes on either end via Corollary \ref{cor:P absorbs colorpure}.  One case is illustrated below, with the $\KC{\cdot}$ notation omitted:
\[
\begin{tikzpicture}[x=2.2em,y=-2.4em,baseline={([yshift=-2ex]current bounding box.center)}]
\smnode at (-.3,0){$\gamma$};
\Bboxst[0]{$P_{n,(\gamma)}$}
\smnode at (-.3,1){$\gamma$};
\Bboxst[1]{$\B^{\ell_0}_{m_0}$}
\smnode at (-.3,2){$\gamma'$};
\Bboxst[2]{$P_{n,(\gamma')}$}
\smnode at (-.3,3){$\gamma'$};
\end{tikzpicture}
\qq{\simeq}
\begin{tikzpicture}[x=2.2em,y=-2.4em,baseline={([yshift=-2ex]current bounding box.center)}]
\smnode at (-.3,0){$\gamma$};
\Bboxst[0]{$P_{n,(\gamma)}$}
\smnode at (-.3,1){$\gamma$};
\Bboxst[1]{$\B^{\ell_0}_{m_0}$}
\smnode at (-.3,2){$\gamma'$};
\Bboxst[2]{$\B^{m_0}_{\ell'_0}$}
\smnode at (-.3,3){$\gamma$};
\Bboxst[3]{$\B^{\ell'_0}_{m'_0}$}
\smnode at (-.3,4){$\gamma'$};
\Bboxst[4]{$P_{n,(\gamma')}$}
\smnode at (-.3,5){$\gamma'$};
\end{tikzpicture}
\qq{\simeq}
\begin{tikzpicture}[x=2.2em,y=-2.4em,baseline={([yshift=-2ex]current bounding box.center)}]
\smnode at (-.3,0){$\gamma$};
\Bboxst[0]{$P_{n,(\gamma)}$}
\smnode at (-.3,1){$\gamma$};
\Bboxst[1]{$\B^{\ell'_0}_{m'_0}$}
\smnode at (-.3,2){$\gamma'$};
\Bboxst[2]{$P_{n,(\gamma')}$}
\smnode at (-.3,3){$\gamma'$};
\end{tikzpicture}.
\]
The equivalence on the left is viewing $\B^{m_0}_{m'_0}$ as color-pure with respect to $\gamma'$, while the equivalence on the right is viewing $\B^{\ell_0}_{\ell'_0}$ as color-pure with respect to $\gamma$.

With finitely many negative crossings in a word $\B$, we use a suitably modified version of Corollary \ref{cor:fin many negs} to get our result.  Finite shifts are also easy to handle by shifting the starting points as well.  Details are left to the reader.
\end{proof}

\begin{remark}
We can apply similar (and simpler) reasoning to assign limiting complexes to positive semi-infinite braid words $\B$ colored from $\gamma$ to $\gamma'$ (that is to say, $\gamma(0)=\gamma$ and $\gamma(a)=\gamma'$ for infinitely many $a$).  If we let $r$ be the smallest index for which $\gamma(r)=\gamma'$, then we can decompose $\B=\B^0_r \B^r_\infty$ and conclude
\[\KC{\B^{(\gamma)}_{(\gamma')}}\simeq \KC{\B^0_r}\otimes P_{n,(\gamma')}.\]
We leave it to the reader to fill in the details, including the passage to positive semi-infinite braids where a Reidemeister move may change the necessary value of $r$, but the limiting complex will remain the same up to chain homotopy equivalence (and degree shifts for Reidemeister II moves) thanks to Lemma \ref{lem:Lim of product} and Corollary \ref{cor:P absorbs colorpure}.
\end{remark}

\begin{remark}
Definition \ref{def:bi-inf braid word} and Corollary \ref{cor:main thm bi-inf} view bi-infinite braids as built `outwards'.  This seems the most natural definition to us, or at least the most amenable to our methods in this paper.  One could also imagine an infinite braid built `inwards', perhaps by defining $\B$ as a `limit' of a sequence of finite words where $\B_{i+1}$ is built by inserting various braids throughout $\B_i$.  As long as such insertions are color-pure, this would preserve the overall coloring leading to a well-defined inverse system and the potential for a limiting complex as above.  However, it is possible to have several non-color-pure insertions that, when taken together, maintain the colors at the endpoints.  It seems unclear whether such a process could also produce an inverse system of maps leading to a well-defined limiting complex.
\end{remark}




\bibliographystyle{alpha}

\bibliography{references}

\end{document}

%% file: BasicFoams.tex
\begin{equation*} 
\xy
(0,0)*{
\begin{tikzpicture} [scale=.6,fill opacity=0.2]
	\path[fill=blue] (2.25,3) to (.75,3) to (.75,0) to (2.25,0);
	\path[fill=red] (.75,3) to [out=225,in=0] (-.5,2.5) to (-.5,-.5) to [out=0,in=225] (.75,0);
	\path[fill=red] (.75,3) to [out=135,in=0] (-1,3.5) to (-1,.5) to [out=0,in=135] (.75,0);	
	\draw [very thick,directed=.55] (2.25,0) to (.75,0);
	\draw [very thick,directed=.55] (.75,0) to [out=135,in=0] (-1,.5);
	\draw [very thick,directed=.55] (.75,0) to [out=225,in=0] (-.5,-.5);
	\draw[very thick, red, directed=.55] (.75,0) to (.75,3);
	\draw [very thick] (2.25,3) to (2.25,0);
	\draw [very thick] (-1,3.5) to (-1,.5);
	\draw [very thick] (-.5,2.5) to (-.5,-.5);
	\draw [very thick,directed=.55] (2.25,3) to (.75,3);
	\draw [very thick,directed=.55] (.75,3) to [out=135,in=0] (-1,3.5);
	\draw [very thick,directed=.55] (.75,3) to [out=225,in=0] (-.5,2.5);
	\node [blue, opacity=1]  at (1.5,2.5) {\tiny{$_{a+b}$}};
	\node[red, opacity=1] at (-.75,3.25) {\tiny{$b$}};
	\node[red, opacity=1] at (-.25,2.25) {\tiny{$a$}};		
\end{tikzpicture}
};
\endxy
\quad \quad \quad \quad
\xy
(0,0)*{
\begin{tikzpicture} [scale=.6,fill opacity=0.2]
	\path[fill=blue] (-2.25,3) to (-.75,3) to (-.75,0) to (-2.25,0);
	\path[fill=red] (-.75,3) to [out=45,in=180] (.5,3.5) to (.5,.5) to [out=180,in=45] (-.75,0);
	\path[fill=red] (-.75,3) to [out=315,in=180] (1,2.5) to (1,-.5) to [out=180,in=315] (-.75,0);	
	\draw [very thick,rdirected=.55] (-2.25,0) to (-.75,0);
	\draw [very thick,rdirected=.55] (-.75,0) to [out=315,in=180] (1,-.5);
	\draw [very thick,rdirected=.55] (-.75,0) to [out=45,in=180] (.5,.5);
	\draw[very thick, red, rdirected=.55] (-.75,0) to (-.75,3);
	\draw [very thick] (-2.25,3) to (-2.25,0);
	\draw [very thick] (1,2.5) to (1,-.5);
	\draw [very thick] (.5,3.5) to (.5,.5);
	\draw [very thick,rdirected=.55] (-2.25,3) to (-.75,3);
	\draw [very thick,rdirected=.55] (-.75,3) to [out=315,in=180] (1,2.5);
	\draw [very thick,rdirected=.55] (-.75,3) to [out=45,in=180] (.5,3.5);
	\node [blue, opacity=1]  at (-1.5,2.5) {\tiny{$_{a+b}$}};
	\node[red, opacity=1] at (.25,3.25) {\tiny{$b$}};
	\node[red, opacity=1] at (.75,2.25) {\tiny{$a$}};		
\end{tikzpicture}
};
\endxy
\end{equation*}
\begin{equation*}\label{ccgens} 
\xy
(0,0)*{
\begin{tikzpicture} [scale=.6,fill opacity=0.2]
	\path[fill=blue] (-.75,4) to [out=270,in=180] (0,2.5) to [out=0,in=270] (.75,4) .. controls (.5,4.5) and (-.5,4.5) .. (-.75,4);
	\path[fill=red] (-.75,4) to [out=270,in=180] (0,2.5) to [out=0,in=270] (.75,4) -- (2,4) -- (2,1) -- (-2,1) -- (-2,4) -- (-.75,4);
	\path[fill=blue] (-.75,4) to [out=270,in=180] (0,2.5) to [out=0,in=270] (.75,4) .. controls (.5,3.5) and (-.5,3.5) .. (-.75,4);
	\draw[very thick, directed=.55] (2,1) -- (-2,1);
	\path (.75,1) .. controls (.5,.5) and (-.5,.5) .. (-.75,1); 
	\draw [very thick, red, directed=.65] (-.75,4) to [out=270,in=180] (0,2.5) to [out=0,in=270] (.75,4);
	\draw[very thick] (2,4) -- (2,1);
	\draw[very thick] (-2,4) -- (-2,1);
	\draw[very thick,directed=.55] (2,4) -- (.75,4);
	\draw[very thick,directed=.55] (-.75,4) -- (-2,4);
	\draw[very thick,directed=.55] (.75,4) .. controls (.5,3.5) and (-.5,3.5) .. (-.75,4);
	\draw[very thick,directed=.55] (.75,4) .. controls (.5,4.5) and (-.5,4.5) .. (-.75,4);
	\node [red, opacity=1]  at (1.5,3.5) {\tiny{$_{a+b}$}};
	\node[blue, opacity=1] at (.25,3.4) {\tiny{$a$}};
	\node[blue, opacity=1] at (-.25,4.1) {\tiny{$b$}};	
\end{tikzpicture}
};
\endxy
\quad \quad \quad \quad
\xy
(0,0)*{
\begin{tikzpicture} [scale=.6,fill opacity=0.2]
	\path[fill=blue] (-.75,-4) to [out=90,in=180] (0,-2.5) to [out=0,in=90] (.75,-4) .. controls (.5,-4.5) and (-.5,-4.5) .. (-.75,-4);
	\path[fill=red] (-.75,-4) to [out=90,in=180] (0,-2.5) to [out=0,in=90] (.75,-4) -- (2,-4) -- (2,-1) -- (-2,-1) -- (-2,-4) -- (-.75,-4);
	\path[fill=blue] (-.75,-4) to [out=90,in=180] (0,-2.5) to [out=0,in=90] (.75,-4) .. controls (.5,-3.5) and (-.5,-3.5) .. (-.75,-4);
	\draw[very thick, directed=.55] (2,-1) -- (-2,-1);
	\path (.75,-1) .. controls (.5,-.5) and (-.5,-.5) .. (-.75,-1); 
	\draw [very thick, red, directed=.65] (.75,-4) to [out=90,in=0] (0,-2.5) to [out=180,in=90] (-.75,-4);
	\draw[very thick] (2,-4) -- (2,-1);
	\draw[very thick] (-2,-4) -- (-2,-1);
	\draw[very thick,directed=.55] (2,-4) -- (.75,-4);
	\draw[very thick,directed=.55] (-.75,-4) -- (-2,-4);
	\draw[very thick,directed=.55] (.75,-4) .. controls (.5,-3.5) and (-.5,-3.5) .. (-.75,-4);
	\draw[very thick,directed=.55] (.75,-4) .. controls (.5,-4.5) and (-.5,-4.5) .. (-.75,-4);
	\node [red, opacity=1]  at (1.25,-1.25) {\tiny{$_{a+b}$}};
	\node[blue, opacity=1] at (-.25,-3.4) {\tiny{$b$}};
	\node[blue, opacity=1] at (.25,-4.1) {\tiny{$a$}};
\end{tikzpicture}
};
\endxy
\end{equation*}
\begin{equation*} \label{uzgens}
\xy
(0,0)*{
\begin{tikzpicture} [scale=.6,fill opacity=0.2]
	\path [fill=red] (4.25,-.5) to (4.25,2) to [out=165,in=15] (-.5,2) to (-.5,-.5) to 
		[out=0,in=225] (.75,0) to [out=90,in=180] (1.625,1.25) to [out=0,in=90] 
			(2.5,0) to [out=315,in=180] (4.25,-.5);
	\path [fill=red] (3.75,.5) to (3.75,3) to [out=195,in=345] (-1,3) to (-1,.5) to 
		[out=0,in=135] (.75,0) to [out=90,in=180] (1.625,1.25) to [out=0,in=90] 
			(2.5,0) to [out=45,in=180] (3.75,.5);
	\path[fill=blue] (.75,0) to [out=90,in=180] (1.625,1.25) to [out=0,in=90] (2.5,0);
	\draw [very thick,directed=.55] (2.5,0) to (.75,0);
	\draw [very thick,directed=.55] (.75,0) to [out=135,in=0] (-1,.5);
	\draw [very thick,directed=.55] (.75,0) to [out=225,in=0] (-.5,-.5);
	\draw [very thick,directed=.55] (3.75,.5) to [out=180,in=45] (2.5,0);
	\draw [very thick,directed=.55] (4.25,-.5) to [out=180,in=315] (2.5,0);
	\draw [very thick, red, directed=.75] (.75,0) to [out=90,in=180] (1.625,1.25);
	\draw [very thick, red] (1.625,1.25) to [out=0,in=90] (2.5,0);
	\draw [very thick] (3.75,3) to (3.75,.5);
	\draw [very thick] (4.25,2) to (4.25,-.5);
	\draw [very thick] (-1,3) to (-1,.5);
	\draw [very thick] (-.5,2) to (-.5,-.5);
	\draw [very thick,directed=.55] (4.25,2) to [out=165,in=15] (-.5,2);
	\draw [very thick, directed=.55] (3.75,3) to [out=195,in=345] (-1,3);
	\node [blue, opacity=1]  at (1.625,.5) {\tiny{$_{a+b}$}};
	\node[red, opacity=1] at (3.5,2.65) {\tiny{$b$}};
	\node[red, opacity=1] at (4,1.85) {\tiny{$a$}};		
\end{tikzpicture}
};
\endxy
\quad \quad \quad \quad
\xy
(0,0)*{
\begin{tikzpicture} [scale=.6,fill opacity=0.2]
	\path [fill=red] (4.25,2) to (4.25,-.5) to [out=165,in=15] (-.5,-.5) to (-.5,2) to
		[out=0,in=225] (.75,2.5) to [out=270,in=180] (1.625,1.25) to [out=0,in=270] 
			(2.5,2.5) to [out=315,in=180] (4.25,2);
	\path [fill=red] (3.75,3) to (3.75,.5) to [out=195,in=345] (-1,.5) to (-1,3) to [out=0,in=135]
		(.75,2.5) to [out=270,in=180] (1.625,1.25) to [out=0,in=270] 
			(2.5,2.5) to [out=45,in=180] (3.75,3);
	\path[fill=blue] (2.5,2.5) to [out=270,in=0] (1.625,1.25) to [out=180,in=270] (.75,2.5);
	\draw [very thick,directed=.55] (4.25,-.5) to [out=165,in=15] (-.5,-.5);
	\draw [very thick, directed=.55] (3.75,.5) to [out=195,in=345] (-1,.5);
	\draw [very thick, red, directed=.75] (2.5,2.5) to [out=270,in=0] (1.625,1.25);
	\draw [very thick, red] (1.625,1.25) to [out=180,in=270] (.75,2.5);
	\draw [very thick] (3.75,3) to (3.75,.5);
	\draw [very thick] (4.25,2) to (4.25,-.5);
	\draw [very thick] (-1,3) to (-1,.5);
	\draw [very thick] (-.5,2) to (-.5,-.5);
	\draw [very thick,directed=.55] (2.5,2.5) to (.75,2.5);
	\draw [very thick,directed=.55] (.75,2.5) to [out=135,in=0] (-1,3);
	\draw [very thick,directed=.55] (.75,2.5) to [out=225,in=0] (-.5,2);
	\draw [very thick,directed=.55] (3.75,3) to [out=180,in=45] (2.5,2.5);
	\draw [very thick,directed=.55] (4.25,2) to [out=180,in=315] (2.5,2.5);
	\node [blue, opacity=1]  at (1.625,2) {\tiny{$_{a+b}$}};
	\node[red, opacity=1] at (3.5,2.65) {\tiny{$b$}};
	\node[red, opacity=1] at (4,1.85) {\tiny{$a$}};		
\end{tikzpicture}
};
\endxy
\end{equation*}
\begin{equation*}\label{MVgensup}
\xy
(0,0)*{
\begin{tikzpicture} [scale=.6,fill opacity=0.2]
	\path[fill=red] (-2.5,4) to [out=0,in=135] (-.75,3.5) to [out=270,in=90] (.75,.25)
		to [out=135,in=0] (-2.5,1);
	\path[fill=blue] (-.75,3.5) to [out=270,in=125] (.29,1.5) to [out=55,in=270] (.75,2.75) 
		to [out=135,in=0] (-.75,3.5);
	\path[fill=blue] (-.75,-.5) to [out=90,in=235] (.29,1.5) to [out=315,in=90] (.75,.25) 
		to [out=225,in=0] (-.75,-.5);
	\path[fill=red] (-2,3) to [out=0,in=225] (-.75,3.5) to [out=270,in=125] (.29,1.5)
		to [out=235,in=90] (-.75,-.5) to [out=135,in=0] (-2,0);
	\path[fill=red] (-1.5,2) to [out=0,in=225] (.75,2.75) to [out=270,in=90] (-.75,-.5)
		to [out=225,in=0] (-1.5,-1);
	\path[fill=red] (2,3) to [out=180,in=0] (.75,2.75) to [out=270,in=55] (.29,1.5)
		to [out=305,in=90] (.75,.25) to [out=0,in=180] (2,0);
	\draw[very thick, directed=.55] (2,0) to [out=180,in=0] (.75,.25);
	\draw[very thick, directed=.55] (.75,.25) to [out=225,in=0] (-.75,-.5);
	\draw[very thick, directed=.55] (.75,.25) to [out=135,in=0] (-2.5,1);
	\draw[very thick, directed=.55] (-.75,-.5) to [out=135,in=0] (-2,0);
	\draw[very thick, directed=.55] (-.75,-.5) to [out=225,in=0] (-1.5,-1);
	\draw[very thick, red, rdirected=.85] (-.75,3.5) to [out=270,in=90] (.75,.25);
	\draw[very thick, red, rdirected=.75] (.75,2.75) to [out=270,in=90] (-.75,-.5);	
	\draw[very thick] (-1.5,-1) -- (-1.5,2);	
	\draw[very thick] (-2,0) -- (-2,3);
	\draw[very thick] (-2.5,1) -- (-2.5,4);	
	\draw[very thick] (2,3) -- (2,0);
	\draw[very thick, directed=.55] (2,3) to [out=180,in=0] (.75,2.75);
	\draw[very thick, directed=.55] (.75,2.75) to [out=135,in=0] (-.75,3.5);
	\draw[very thick, directed=.65] (.75,2.75) to [out=225,in=0] (-1.5,2);
	\draw[very thick, directed=.55]  (-.75,3.5) to [out=225,in=0] (-2,3);
	\draw[very thick, directed=.55]  (-.75,3.5) to [out=135,in=0] (-2.5,4);
	\node[red, opacity=1] at (-2.25,3.375) {\tiny$c$};
	\node[red, opacity=1] at (-1.75,2.75) {\tiny$b$};	
	\node[red, opacity=1] at (-1.25,1.75) {\tiny$a$};
	\node[blue, opacity=1] at (0,2.75) {\tiny$_{b+c}$};
	\node[blue, opacity=1] at (0,.25) {\tiny$_{a+b}$};
	\node[red, opacity=1] at (1.37,2.5) {\tiny$_{a+b+c}$};	
\end{tikzpicture}
};
\endxy
\quad \quad \quad \quad
\xy
(0,0)*{
\begin{tikzpicture} [scale=.6,fill opacity=0.2]
	\path[fill=red] (-2.5,4) to [out=0,in=135] (.75,3.25) to [out=270,in=90] (-.75,.5)
		 to [out=135,in=0] (-2.5,1);
	\path[fill=blue] (-.75,2.5) to [out=270,in=125] (-.35,1.5) to [out=45,in=270] (.75,3.25) 
		to [out=225,in=0] (-.75,2.5);
	\path[fill=blue] (-.75,.5) to [out=90,in=235] (-.35,1.5) to [out=315,in=90] (.75,-.25) 
		to [out=135,in=0] (-.75,.5);	
	\path[fill=red] (-2,3) to [out=0,in=135] (-.75,2.5) to [out=270,in=125] (-.35,1.5) 
		to [out=235,in=90] (-.75,.5) to [out=225,in=0] (-2,0);
	\path[fill=red] (-1.5,2) to [out=0,in=225] (-.75,2.5) to [out=270,in=90] (.75,-.25)
		to [out=225,in=0] (-1.5,-1);
	\path[fill=red] (2,3) to [out=180,in=0] (.75,3.25) to [out=270,in=45] (-.35,1.5) 
		to [out=315,in=90] (.75,-.25) to [out=0,in=180] (2,0);				
	\draw[very thick, directed=.55] (2,0) to [out=180,in=0] (.75,-.25);
	\draw[very thick, directed=.55] (.75,-.25) to [out=135,in=0] (-.75,.5);
	\draw[very thick, directed=.55] (.75,-.25) to [out=225,in=0] (-1.5,-1);
	\draw[very thick, directed=.45]  (-.75,.5) to [out=225,in=0] (-2,0);
	\draw[very thick, directed=.35]  (-.75,.5) to [out=135,in=0] (-2.5,1);	
	\draw[very thick, red, rdirected=.75] (-.75,2.5) to [out=270,in=90] (.75,-.25);
	\draw[very thick, red, rdirected=.85] (.75,3.25) to [out=270,in=90] (-.75,.5);
	\draw[very thick] (-1.5,-1) -- (-1.5,2);	
	\draw[very thick] (-2,0) -- (-2,3);
	\draw[very thick] (-2.5,1) -- (-2.5,4);	
	\draw[very thick] (2,3) -- (2,0);
	\draw[very thick, directed=.55] (2,3) to [out=180,in=0] (.75,3.25);
	\draw[very thick, directed=.55] (.75,3.25) to [out=225,in=0] (-.75,2.5);
	\draw[very thick, directed=.55] (.75,3.25) to [out=135,in=0] (-2.5,4);
	\draw[very thick, directed=.55] (-.75,2.5) to [out=135,in=0] (-2,3);
	\draw[very thick, directed=.55] (-.75,2.5) to [out=225,in=0] (-1.5,2);
	\node[red, opacity=1] at (-2.25,3.75) {\tiny$c$};
	\node[red, opacity=1] at (-1.75,2.75) {\tiny$b$};	
	\node[red, opacity=1] at (-1.25,1.75) {\tiny$a$};
	\node[blue, opacity=1] at (-.125,2.25) {\tiny$_{a+b}$};
	\node[blue, opacity=1] at (-.125,.75) {\tiny$_{b+c}$};
	\node[red, opacity=1] at (1.35,2.75) {\tiny$_{a+b+c}$};	
\end{tikzpicture}
};
\endxy
\end{equation*}